\documentclass[12pt]{article}

\usepackage{amsmath, amssymb}
\usepackage{amsthm}
\usepackage{enumitem}

\usepackage{xcolor}
\definecolor{darkred}{RGB}{139,0,0}
\definecolor{darkgreen}{RGB}{0,120,0}
\definecolor{darkmagenta}{RGB}{139,0,139}
\definecolor{darkpurple}{RGB}{110,0,180}
\definecolor{darkblue}{RGB}{40,0,200}
\definecolor{darkorange}{RGB}{255,140,0}

\newcommand{\fp}[1]{\textcolor{darkmagenta}{#1}}

\newcommand{\gl}[1]{\textcolor{darkgreen}{GL: #1}}

\newcommand{\bsx}{{\boldsymbol{x}}}
\newcommand{\bsh}{\boldsymbol{h}}

\newcommand{\bszero}{\boldsymbol{0}}
\newcommand{\bsgamma}{\boldsymbol{\gamma}}

\newcommand{\bsk}{\boldsymbol{k}}
\newcommand{\bsl}{\boldsymbol{l}}
\newcommand{\bsone}{\boldsymbol{1}}
\newcommand{\bsy}{\boldsymbol{y}}

\newcommand{\rd}{\,{\rm d}}

\newcommand{\RR}{\mathbb{R}}
\newcommand{\NN}{\mathbb{N}}
\newcommand{\ZZ}{\mathbb{Z}}
\newcommand{\CC}{\mathbb{C}}

\newcommand{\cA}{\mathcal{A}}

\newcommand{\cO}{\mathcal{O}}
\newcommand{\cH}{\mathcal{H}}
\newcommand{\icomp}{\mathtt{i}}
\newcommand{\id}{\operatorname{id}}

\theoremstyle{plain}
\newtheorem{theorem}{Theorem}
\newtheorem{proposition}[theorem]{Proposition}
\newtheorem{lemma}[theorem]{Lemma}
\newtheorem{corollary}[theorem]{Corollary}
\theoremstyle{definition}
\newtheorem{definition}{Definition}
\newtheorem{remark}[theorem]{Remark}
\newtheorem{example}[theorem]{Example}

\newcommand{\APP}{{\rm APP}}
\newcommand{\INT}{{\rm INT}}
\newcommand{\uu}{{\mathfrak u}}

\newcommand{\bseta}{\boldsymbol{\eta}}
\newcommand{\bstau}{\boldsymbol{\tau}}
\newcommand{\tra}{{\rm trace}}
\DeclareMathOperator*{\esssup}{ess\,sup}

\allowdisplaybreaks

\begin{document}

\title{Tractability of $L_2$-approximation and integration in weighted Hermite spaces of finite smoothness}

\author{Gunther Leobacher, Friedrich Pillichshammer and Adrian Ebert\thanks{The authors are supported by the Austrian Science Fund (FWF), Projects F5508-N26 (Leobacher), F5509-N26 (Pillichshammer), and F5506-N26 (Ebert), which are parts of the Special Research Program ``Quasi-Monte Carlo Methods: Theory and Applications''.}}

\date{\today}

\maketitle

\begin{abstract}
In this paper we consider integration and $L_2$-approximation for functions over $\RR^s$ from weighted Hermite spaces. The first part of the paper is devoted to a comparison of several weighted Hermite spaces that appear in literature, which is interesting on its own. Then we study tractability of the integration and $L_2$-approximation problem for the introduced Hermite spaces, which describes the growth rate of the information complexity when the error threshold $\varepsilon$ tends to 0 and the problem dimension $s$ grows to infinity. Our main results are characterizations of tractability in terms of the involved weights, which  model the importance of the successive coordinate directions for functions from the weighted Hermite spaces.
\end{abstract}

\section{Introduction}\label{intro}

Weighted integration and approximation of functions over the whole $s$-dimensional Euclidean  space $\RR^s$ appear in many practical problems, often with respect to the Gaussian weight $\varphi$. From a theoretical point of view Gaussian problems can be studied in a very elegant way in the context of Hermite spaces of functions, which are the major object of interest of this paper. We present several examples of weighted Hermite spaces that appear in literature and discuss relations, similarities but also differences between these spaces. In order to be able to go into more details we briefly introduce the general function space setting.

We consider weighted Hermite spaces of functions with finite smoothness, using a similar notation as in \cite{DILP18}. 
In particular, for $k \in \NN_0$, we denote the $k$-th Hermite polynomial by
\begin{equation*}
H_k(x) = \frac{(-1)^k}{\sqrt{k!}} \exp(x^2/2) \frac{\rd^k}{\rd x^k} \exp(-x^2/2).
\end{equation*}
For example, 
\begin{align*}
H_0(x)=1,\ H_1(x)=x,\ H_2(x)=\tfrac{1}{\sqrt{2}}(x^2-1),\ H_3(x)=\tfrac{1}{\sqrt{6}}(x^3-3x), \ldots . 
\end{align*}
Here we follow the definition given in \cite{B98}, but we remark that there are slightly different ways to introduce Hermite polynomials (see, e.g., \cite{szeg}).  We recall the definition of the standard normal density as $\varphi(x)=\frac{1}{\sqrt{2 \pi}} \exp(-x^2/2)$ for $x \in \RR$. Furthermore, for $s \in \NN$, $\bsk = (k_1,\ldots,k_s) \in \NN_0^s$ and 
$\bsx = (x_1,\ldots,x_s) \in \RR^s$ we define the $\bsk$-th Hermite polynomial by
\begin{equation*}
	H_{\bsk}(\bsx)
	:=
	\prod_{j=1}^{s} H_{k_j}(x_j)
\end{equation*}
and additionally set $\varphi_s(\bsx) := \prod_{j=1}^s \varphi(x_j)$, i.e., $\varphi_s$ is the standard normal density on $\RR^s$. It is well known, see \cite{B98}, that the sequence of
Hermite polynomials $(H_{\bsk})_{\bsk \in \NN_0^s}$ forms an orthonormal basis of the function space $L_2(\RR^s,\varphi_s)$, i.e., for all $f\in L_2(\RR^s,\varphi_s)$ we have the \textit{Hermite expansion}
\[
f\sim 
\sum_{\bsk\in \NN_0^s}
\widehat f(\bsk)H_{\bsk}\,,
\]
where $\sim$ denotes convergence in $L_2(\RR^s,\varphi_s)$ and where 
\begin{align*}
\widehat{f}(\bsk)=\int_{\RR^s} f(\bsx) H_{\bsk}(\bsx)
\varphi_s(\bsx)\rd \bsx
\end{align*}
is the $\bsk$-th \textit{Hermite coefficient} of $f$.

Similar to what has been done in~\cite{IL}, we are now going to
define function spaces based on Hermite expansions. These spaces
are Hilbert spaces with a \textit{reproducing kernel}. For details
on reproducing kernel Hilbert spaces, we refer to the classical treatment~\cite{aronszajn50}.

For the time being, let $R: \NN_0^s \rightarrow \RR^+$ be a summable function, i.e.,
$\sum_{\bsk\in\NN_0^s} R(\bsk) < \infty$ (this condition will be slightly relaxed later on in concrete examples). Define a so-called {\it Hermite kernel} as
\begin{align}\label{def:Hkernel}
K_{R}(\bsx,\bsy)=\sum_{\bsk \in \NN_0^s} R(\bsk) H_{\bsk}(\bsx)
H_{\bsk}(\bsy)\ \ \ \ \mbox{ for }\ \ \bsx,\bsy \in \RR^s
\end{align}
and an inner product
\begin{align}\label{def:iprod}
\langle f,g\rangle_{R} =\sum_{\bsk \in \NN_0^s}
\frac{1}{R(\bsk)}\, \widehat{f}(\bsk) \widehat{g}(\bsk)\,.
\end{align}
The weight coefficients $R(\bsk)$ are sometimes also referred to as {\it Fourier weights} (see \cite[p.~3]{GHHR}). Note that $K_R(\bsx,\bsy)$ is well defined for all $\bsx,\bsy \in \RR^s$, since
\begin{align*}
|K_R(\bsx,\bsy)| \le \sum_{\bsk \in \NN_0^s}  R(\bsk) |H_{\bsk}(\bsx)| \,
|H_{\bsk}(\bsy)| \le \frac{1}{\sqrt{\varphi_s(\bsx)
\varphi_s(\bsy)}}\sum_{\bsk \in \NN_0^s}  R(\bsk)< \infty,
\end{align*}
where we have used Cramer's bound for Hermite polynomials, see, e.g., \cite[p.~324]{sansone}, which states that
\begin{align*}
\vert H_{k}(x)\vert\leq
\frac{1}{\sqrt{\varphi(x)}}\quad\textnormal{ for all } k\in\NN_0.
\end{align*}

Let $\mathcal{H}(K_R)$ be the reproducing kernel Hilbert space
corresponding to $K_R$. Such spaces are typically known as {\it Hermite spaces} (see \cite[Definition~3.4]{GHHR}).
The norm in $\mathcal{H}(K_R)$ is given by $\| f
\|_R^2 =\langle f , f \rangle_R$. From this we see
that the functions in $\mathcal{H}(K_R)$ are characterized by the
decay rate of their Hermite coefficients, which is regulated by
the function $R$. Roughly speaking, the faster $R$ decreases as
$\bsk$ moves away from the origin, the faster the Hermite coefficients of the elements
of $\mathcal{H}(K_R)$ decrease. 

It is worth mentioning the similarity of Hermite spaces to Korobov spaces, the elements of which are 
$\CC$-valued continuous periodic functions on the unit interval with a prescribed convergence speed of the Fourier coefficients (see, for example, \cite{DKP22} for detailed information).
The norm and kernel on a Korobov space are obtained from their analogs by replacing Hermite coefficients 
by Fourier coefficients, Hermite polynomials by the functions $x\mapsto {\rm e}^{2\pi \icomp k x}$, $k\in \ZZ$, and
summation over the non-negative integers by summation over all integers.
However, usually the term ``Korobov spaces'' is interpreted in a more narrow sense, where 
the Fourier weights are of the form 
\[
R_{\text{Kor},\alpha,\gamma}(k)
=\begin{cases}
1 & \text{ if } k=0,\\
\gamma |k|^{-\alpha} & \text{ if } k\ne 0,
\end{cases}
\]
for some non-negative weight $\gamma$ and a smoothness parameter $\alpha> 1$.

We are interested in integration and $L_2$-approximation of functions from Hermite spaces. In \cite{IL}, the case of polynomially decreasing $R$ as well as exponentially decreasing $R$ was considered. In \cite{IKLP14,IKPW16a,IKPW16b} further results were obtained for numerical integration and/or $L_2$-approximation for exponentially decreasing $R$. In this case exponential convergence rates can be achieved as well as several notions of tractability which exactly describe a favorable dependence of the errors on the dimension.

Numerical integration for the case of polynomially decaying Fourier weights $R$ is  considered further in \cite{DILP18}. The main focus there is in achieving optimal error convergence rates for the worst-case error leaving aside the exact analysis of the dependence of the errors on the dimension $s$.

In this paper, we continue the work on polynomially decreasing $R$ for $L_2$-approximation and integration in the worst-case setting, where the focus will be on very high-dimensional problems. The quantity of interest is the information complexity which is the number of information evaluations required in order to push the worst-case error below a given error threshold $\varepsilon$, where $\varepsilon \in (0,1)$. The important question that arises for applications is how this information complexity depends on $\varepsilon$ and on the dimension $s$. This question is the subject of tractability theory (see the trilogy \cite{NW08,NW10,NW12} by Novak and Wo\'{z}niakowski for general information). Tractability is a concept to characterize the growth rate of the information complexity when $\varepsilon$ tends to 0 and $s$ grows to infinity. We study tractability for $L_2$-approximation and integration in weighted Hermite spaces and give conditions for various notions of tractability  in terms of the involved weights $\bsgamma=(\gamma_j)_{j \ge 1}$ that model the ``importance'' of the successive coordinate directions.  Despite the apparent similarity between Hermite- and Korobov spaces, there is much more known about tractability of approximation in the worst-case setting for the latter. See \cite{EP21} for matching necessary and sufficient conditions for both standard and linear information. Thus our aim here is to close some of the gaps in knowledge about Hermite spaces.

The paper is organized as follows. First, in Section~\ref{Hspaces}, we discuss and compare several possibilities of describing finite smoothness via various choices of Fourier weights $R$ that appear in literature. This section is  interesting on its own, since often it is not clear which is the right choice of a Hermite space for a given problem. However, we will see that the proposed spaces are equivalent as normed function spaces. For the main example, the so-called {\em Gaussian ANOVA space}, we present an integral representation of the reproducing kernel in Theorem \ref{pr:intrep}.

In Section~\ref{sec:int_app_general} we present the general $L_2$-approximation and integration problem for Hermite spaces and discuss some general facts and relations.

In Section~\ref{sec:app} we will study tractability properties of $L_2$-approximation for functions from a Hermite space for permissible information class from $\Lambda^{{\rm all}}$, consisting of arbitrary  linear functionals, and from $\Lambda^{{\rm std}}$, consisting exclusively of functions evaluations. The main results are Theorem~\ref{thm:main} and Corollary~\ref{cor:app:R:std} (for $\Lambda^{{\rm all}}$) and Theorem~\ref{thm:tract-standard}  (for $\Lambda^{{\rm std}}$). While for $\Lambda^{{\rm all}}$ we get a very clear picture of the whole situation,  that is, we have both necessary and sufficient conditions for a range of notions of tractability,  for $\Lambda^{{\rm std}}$ necessary conditions
remain open problems.    

Tractability for the integration problem is discussed in Section~\ref{sec:int}. Here the main result is Theorem~\ref{thm:main:int}, giving sufficient conditions for several notions of tractability. 

\section{Weighted Hermite spaces of finite smoothness}\label{Hspaces}

Like for  the case of Sobolev spaces of smooth functions over $[0,1]^s$ (see \cite[Appendix~A]{NW08}) there are various possible ways for introducing Hermite spaces of functions with finite smoothness over $\RR^s$. We consider the weighted setting and discuss possible choices for the Fourier weights $R$. Throughout let $\alpha \ge 1$ be a parameter that will describe the smoothness via the decay rate of the Hermite coefficients of a function to zero. If $\alpha \in \NN$ in many cases this can be related to the smoothness of functions with respect to the existence and integrability of partial derivatives of functions. 

\subsection{A Gaussian ANOVA space} \label{sec:ANOVA}

Our first example will be our main object of interest. Later on we will study approximation and integration of functions from this space.

Let $\alpha \ge 1$ and let $\bsgamma=(\gamma_j)_{j \ge 1}$ be a sequence of so-called product weights. We assume throughout that the weights are in $(0,1]$ and that they are in descending order, i.e.,  $1 \ge \gamma_1 \ge \gamma_2 \ge \gamma_3 \ge \ldots >0$. Then the function space of interest is the reproducing kernel Hilbert space 
$\cH_{r_{s,\alpha,\bsgamma}}$ with kernel  \eqref{def:Hkernel} and corresponding inner product \eqref{def:iprod} determined by  $R(\bsk)=r_{s,\alpha,\bsgamma}(\bsk) := \prod_{j=1}^s r_{\alpha,\gamma_j}(k_j)$ with \begin{equation*}
	r_{\alpha,\gamma}(k)
	:=
	\left\{\begin{array}{ll}
		1 & \text{for } k=0, \\[0.5em]
		\gamma \frac{1}{k!}  & \text{for } 1 \le k < \alpha,\\[0.5em]
		\gamma \frac{(k-\alpha)!}{k!} & \text{for } k \ge \alpha.
	\end{array}\right.
\end{equation*}
for a generic weight $\gamma \in (0,1]$. Note that we always have $r_{\alpha,\gamma}(k) \in (0,1]$.

The space $\cH_{r_{s,\alpha,\bsgamma}}:=\cH(K_{r_{s,\alpha,\bsgamma}})$ is a {\it weighted Hermite space} with smoothness parameter $\alpha$ (see Equation~\eqref{eq:norm} below) and weights $\bsgamma$. The weights are introduced in order to model the ``importance'' of the different coordinates for the functions from the space, where weight $\gamma_j$ is assigned to coordinate direction $j \in \NN$, according to an idea of Sloan and Wo\'{z}niakowski (see \cite{slowo}). If all weights equal 1, i.e., if $\gamma_j=1$ for all $j \in \NN$, then we speak about the unweighted Hermite space.

The following lemma gives easy bounds on the decay of the function $r_{\alpha,\gamma}$, showing that $r_{\alpha,\gamma}$ has the same decay rate as the corresponding Fourier weights for the classical Korobov space of smoothness $\alpha$.

\begin{lemma}\label{le:bdrk}
For all $k \in \NN$ we have
\begin{equation*} 
\frac{\gamma}{k^\alpha} \le r_{\alpha,\gamma}(k) \le \gamma \left(\frac{\alpha}{k}\right)^\alpha.
\end{equation*}   
\end{lemma}

\begin{proof}
If $1 \le k < \alpha$ we have $$\frac{1}{r_{\alpha,\gamma}(k)} = \frac{1}{\gamma} k! \le  \frac{1}{\gamma} k^k \le \frac{1}{\gamma} k^{\alpha}.$$ If $k \ge \alpha$ we have $$\frac{1}{r_{\alpha,\gamma}(k)} = \frac{1}{\gamma} \frac{k!}{(k-\alpha)!}=\frac{1}{\gamma} k (k-1) \cdots (k-\alpha+1) \le   \frac{1}{\gamma} k^{\alpha}.$$ Hence we find that
\begin{equation*} 
	\frac{\gamma}{k^\alpha} \le 	r_{\alpha,\gamma}(k)	. 
\end{equation*}
In order to show the upper bound we consider the case that $k> \alpha$ first. Then  
\begin{equation*}
	r_{\alpha,\gamma}(k) 
	= 
	\frac{\gamma}{k(k-1) \cdots (k-\alpha +1)} \le \frac{\gamma}{(k-\alpha+1)^\alpha}=\frac{\gamma}{k^\alpha (1-\frac{\alpha -1}{k})^\alpha} 
	\le 
	\gamma \left(\frac{\alpha}{k}\right)^\alpha
	, 
\end{equation*}
because for $k>\alpha$ we have 
\begin{equation*}
	1 - \frac{\alpha-1}{k} 
	\ge 
	1 - \frac{\alpha-1}{\alpha}
	= 
	\frac{1}{\alpha}
	. 
\end{equation*}
For $1 \le k \le \alpha$ we have $r_{\alpha,\gamma}(k)=
	\frac{\gamma}{k!}$ and $\left(\frac{\alpha}{k}\right)^\alpha \ge 1$, and hence 
\begin{equation*}
	r_{\alpha,\gamma}(k) 
	 \le \gamma \left(\frac{\alpha}{k}\right)^\alpha
	.
\end{equation*}
This finishes the proof.
\end{proof}

Note that for $\alpha=1$ we have $\sum_{k\in \NN_0}r_{\alpha,\gamma}(k)=\infty$. Nevertheless, from \cite[Lemma~1]{DILP18} we know that for all $\bsk \in \NN_0$ and for all $\bsx \in \RR^s$ we even have $$|H_{\bsk}(\bsx) \sqrt{\varphi_s(\bsx)}| \le \prod_{j=1}^s \min\left(1,\frac{\sqrt{\pi}}{k_j^{1/12}}\right).$$ This is a slight improvement of Cramer's bound mentioned earlier in this paper. From this estimate it follows again that $K_{s,\alpha,\bsgamma}(\bsx,\bsy)$ is well defined for all $\alpha\ge 1$ and for all $\bsx,\bsy \in \RR^s$, since
\begin{eqnarray*}
|K_{r_{s,\alpha,\bsgamma}}(\bsx,\bsy)| & \le & \sum_{\bsk \in \NN_0^s} r_{s,\alpha,\bsgamma}(\bsk) |H_{\bsk}(\bsx) H_{\bsk}(\bsy)|\\
& \le & \frac{1}{\sqrt{\varphi_s(\bsx)\varphi_s(\bsy)}}\sum_{\bsk \in \NN_0^s} r_{s,\alpha,\bsgamma}(\bsk) \prod_{j=1}^s \min\left(1,\frac{\pi}{k_j^{1/6}}\right)\\
& \le & \frac{1}{\sqrt{\varphi_s(\bsx)\varphi_s(\bsy)}} \prod_{j=1}^s \left(1+\gamma_j \alpha^{\alpha} \left(\sum_{1<k<\pi ^6} \frac{1}{k^{\alpha}} +\pi \sum_{k\ge \pi^6} \frac{1}{k^{\alpha +1/6}}\right)\right) <  \infty.
\end{eqnarray*}

Now we explain how the parameter $\alpha$ is related to the smoothness of the functions from the Hermite space $\cH_{r_{s,\alpha,\bsgamma}}$ whenever $\alpha$ is an integer. Let $\alpha \in \NN$. For $f\in\cH_{r_{s,\alpha,\bsgamma}}$ we have the Hermite expansion, see \cite{IL},
\begin{align*}
f(\bsx)=\sum_{\bsk\in\NN_0^s}\widehat{f}(\bsk)H_{\bsk}(\bsx)\qquad\text{for all }\bsx\in\RR^s
\end{align*}
and for any $\bstau=(\tau_1,\ldots,\tau_s)\in\NN_0^s$ with $\bstau \le \alpha $ we have that
\begin{align}\label{eq:partabl}
\partial_{\bsx}^{\bstau}f  \sim \sum_{\bsk\geq\bstau}\widehat{f}(\bsk)\sqrt{\frac{\bsk!}{(\bsk-\bstau)!}}\,H_{\bsk-\bstau}.
\end{align}
For $s \in \NN$ we write $[s]:=\{1,2,\ldots,s\}$. Observe also the use of the standard multiindex notation 
$\bstau!=\prod_{j=1}^s\tau_j!$ and $\bstau\le \alpha$, which means that $\tau_j \le \alpha$ for all $j\in [s]$ for $\bstau\in\NN_0^s$ and likewise $\bstau \le \bsk$, which means that $\tau_j \le k_j$ for all $j \in [s]$, for $\bstau,\bsk \in \NN_0^s$. Then the inner product of the weighted Hermite space $\cH_{r_{s,\alpha,\bsgamma}}$ can be written as

\begin{align}\label{eq:inprod}
\langle f, g \rangle_{r_{s,\alpha,\bsgamma}}&=\sum_{\uu\subseteq[s]}\sum_{\bstau_{\uu}\in\{0,\ldots,\alpha-1\}^{|\uu|}}\gamma_{\bstau_\uu}^{-1} \int_{\RR^{s-|\uu|}}\left(\int_{\RR^{|\uu|}}\partial_{\bsx}^{(\bstau_{\uu},\alpha_{-\uu})}f(\bsx)\varphi_{|\uu|}(\bsx_{\uu})\rd\bsx_{\uu}\right)\nonumber\\
&\qquad\times\left(\int_{\RR^{|\uu|}}\partial_{\bsx}^{(\bstau_{\uu},\alpha_{-\uu})}g(\bsx)\varphi_{|\uu|}(\bsx_{\uu})\rd\bsx_{\uu}\right)\varphi_{s-|\uu|}(\bsx_{-\uu})\rd\bsx_{-\uu},
\end{align}
where $(\bstau_{\uu},\alpha_{-\uu})\in \NN_0^s$ denotes the multiindex  for which the $j$-th component equals  $\alpha$ for $j\notin\uu$ and $\tau_j$ for $j\in\uu$,   and where $\gamma_{\bstau_{\uu}}$ is the product of the $\gamma_j$ over those $j$ for which the $j$-th component of $(\bstau_{\uu},\alpha_{-\uu})$ does not equal  $0$, i.e.,
\begin{align*}
\gamma_{\bstau_{\uu}}=\prod_{\substack{j =1\\ \tau_j\neq0\vee j\notin\uu}}^s\gamma_j=\left(\prod_{j \in [s] \setminus \uu} \gamma_j\right) \prod_{\substack{j \in \uu \\ \tau_j \not=0}} \gamma_j,
\end{align*} 
and $\partial_{\bsx}^{\bseta}=\frac{\partial^{\eta_1}}{(\partial x_1)^{\eta_1}}\cdots\frac{\partial^{\eta_s}}{(\partial x_s)^{\eta_s}}$ for $\bseta=(\eta_1,\ldots,\eta_s)\in\NN_0^s$.
Hence we may express the norm in $\cH_{r_{s,\alpha,\bsgamma}}$ as a certain instance of a Sobolev type norm in the form
\begin{align}\label{eq:norm}
\|f\|_{r_{s,\alpha,\bsgamma}}^2 &=\sum_{\uu\subseteq[s]}\sum_{\bstau_{\uu}\in\{0,\ldots,\alpha-1\}^{|\uu|}}\gamma_{\bstau_{\uu}}^{-1} \int_{\RR^{s-|\uu|}}\left(\int_{\RR^{|\uu|}}\partial_{\bsx}^{(\bstau_{\uu},\alpha_{-\uu})}f(\bsx)\varphi_{|\uu|}(\bsx_{\uu})\rd\bsx_{\uu}\right)^2\varphi_{s-|\uu|}(\bsx_{-\uu})\rd\bsx_{-\uu}\,,
\end{align}
where $-\uu:=[s]\setminus \uu$. We will provide a proof of \eqref{eq:norm} (and thus of \eqref{eq:inprod}) shortly. 
In particular, a finite norm for $f \in \cH_{r_{s,\alpha,\bsgamma}}$ requires that all partial mixed derivatives of $f$ of order $\alpha$ in every coordinate direction are  square integrable. Actually, as a vector space,  $\cH_{r_{s,\alpha,\bsgamma}}$ is precisely the space of continuous functions on  $\RR^s$, for which for every $\bstau\le \alpha$ the $\bstau$-th mixed weak partial derivative exists and is square integrable. 

In order to have a concrete impression we describe the following special instances.
 
\begin{example}\label{ex1}\rm
For example, for $s=\alpha=1$ we have
$$\|f\|_{r_{1,1,\gamma}}^2 = \left(\int_{\RR} f(x) \varphi(x)\rd x\right)^2+\frac{1}{\gamma} \int_{\RR} (f'(x))^2 \varphi(x) \rd x,$$
for $s=1$, $\alpha \in \NN$ we have
\begin{equation}\label{fo_norm1a}
\|f\|_{r_{1,\alpha,\gamma}}^2 = \left(\int_{\RR} f(x) \varphi(x)\rd x\right)^2+\frac{1}{\gamma} \sum_{\tau=1}^{\alpha-1} \left(\int_{\RR} f^{(\tau)}(x) \varphi(x) \rd x\right)^2+\frac{1}{\gamma} \int_{\RR}(f^{(\alpha)}(x))^2 \varphi(x) \rd x,
\end{equation}
and for $s=2$, $\alpha=1$ we have
\begin{align*}
\|f\|_{r_{2,1,\gamma}}^2 = & \left( \int_{\RR^2} f(x_1,x_2) \varphi(x_1)\varphi(x_2) \rd (x_1,x_2)\right)^2 \\
& +\frac{1}{\gamma_1} \int_{\RR} \left(\int_{\RR} \frac{\partial f(x_1,x_2)}{\partial x_1} \varphi(x_2) \rd x_2 \right)^2 \varphi(x_1) \rd x_1 \\
& +\frac{1}{\gamma_2} \int_{\RR} \left(\int_{\RR} \frac{\partial f(x_1,x_2)}{\partial x_2} \varphi(x_1) \rd x_1 \right)^2 \varphi(x_2) \rd x_2 \\
& + \frac{1}{\gamma_1 \gamma_2} \int_{\RR^2} \left(\frac{\partial^2 f(x_1,x_2)}{\partial x_1 \partial x_2}\right)^2 \varphi(x_1)\varphi(x_2) \rd (x_1,x_2).
\end{align*} 

\end{example}

\begin{proof}[Proof of Equation~\eqref{eq:norm} ]
We  introduce another multiindex notation: $(\bsk-\bstau)_{\uu}:=\bsk_{\uu}-\bstau_{\uu}$ resp.~$(\bsk-\alpha)_{-\uu}:=\bsk_{-\uu}-\alpha_{-\uu}$. With this we write
\begin{align*}
\partial_\bsx^{(\bstau_\uu,\alpha_{-\uu})} f(\bsx)
&\sim \sum_{\bsk\ge (\bstau_\uu,\alpha_{-\uu})}\sqrt{\frac{\bsk!}{(\bsk-(\bstau_\uu,\alpha_{-\uu}))!}}\widehat f(\bsk)H_{\bsk-(\bstau_\uu,\alpha_{-\uu})}(\bsx)\\
&=\sum_{\bsk_{-\uu}\ge \alpha_{-\uu}}\sum_{\bsk_\uu\ge \bstau_\uu}
\sqrt{\frac{\bsk_{-\uu}!}{(\bsk-\alpha)_{-\uu}!}}\sqrt{\frac{\bsk_\uu!}{(\bsk-\bstau)_\uu!}}\widehat f(\bsk)
H_{(\bsk-\alpha)_{-\uu}}(\bsx_{-\uu})H_{(\bsk-\bstau)_\uu}(\bsx_\uu).
\end{align*}
Thus 
\begin{align*}
\lefteqn{\left(\int_{\RR^{|\uu|}}\partial_\bsx^{(\bstau_\uu,\alpha_{-\uu})} f(\bsx)\varphi_{|\uu|}(\bsx_\uu)d\bsx_\uu\right)^2}\\
&=\left(\sum_{\bsk_{-\uu}\ge \alpha_{-\uu}}
\sqrt{\frac{\bsk_{-\uu}!}{(\bsk-\alpha)_{-\uu}!}}\sqrt{\bstau_\uu!}\,
\widehat f(\bstau_\uu,\bsk_{-\uu})
H_{(\bsk-\alpha)_{-\uu}}(\bsx_{-\uu})\right)^2\\
&=\bstau_\uu!\sum_{\bsk_{-\uu}\ge \alpha_{-\uu}}\sum_{\bsl_{-\uu}\ge \alpha_{-\uu}}
\sqrt{\frac{\bsk_{-\uu}!}{(\bsk-\alpha)_{-\uu}!}}\sqrt{\frac{\bsl_{-\uu}!}{(\bsl-\alpha)_{-\uu}!}}
\widehat f(\bstau_\uu,\bsk_{-\uu})
H_{(\bsk-\alpha)_{-\uu}}(\bsx_{-\uu})H_{(\bsl-\alpha)_{-\uu}}(\bsx_{-\uu})
\end{align*}
such that, using the orthogonality of the multidimensional Hermite polynomials,
\begin{align*}
\lefteqn{\int_{\RR^{s-|\uu|}}\left(\int_{\RR^{|\uu|}}\partial_\bsx^{(\bstau_\uu,\alpha_{-\uu})} f(\bsx)\varphi_{|\uu|}(\bsx_\uu)d\bsx_\uu\right)^2\varphi_{s-|\uu|}(\bsx_{-\uu})\rd\bsx_{-\uu}}\\
&=\bstau_\uu!\sum_{\bsk_{-\uu}\ge \alpha_{-\uu}}
\frac{\bsk_{-\uu}!}{(\bsk-\alpha)_{-\uu}!}
\widehat f(\bstau_\uu,\bsk_{-\uu})^2.
\end{align*}
With this we finally get 
\begin{align*}
\lefteqn{\sum_{\uu\subseteq[s]}\sum_{\bstau_\uu\in \{0\,\ldots,\alpha-1\}^{|\uu|}}\bsgamma_{\bstau_\uu}^{-1}\int_{\RR^{s-|\uu|}}\left(\int_{\RR^{|\uu|}}\partial_\bsx^{(\bstau_\uu,\alpha_{-\uu})} f(\bsx)\varphi_{|\uu|}(\bsx_\uu)d\bsx_\uu\right)^2\varphi_{s-|\uu|}(\bsx_{-\uu})\rd\bsx_{-\uu}}\\
&=\sum_{\uu\subseteq[s]}\sum_{\bstau_\uu\in \{0\,\ldots,\alpha-1\}^{|\uu|}}\sum_{\bsk_{-\uu}\ge \alpha_{-\uu}}
\Bigg(\prod_{\substack{j\in \uu\\\tau_j\ne 0}}\gamma_j^{-1}\Bigg)\bstau_\uu!\Bigg(\prod_{\ell\in [s]\setminus \uu}\gamma_\ell^{-1}\Bigg)
\frac{\bsk_{-\uu}!}{(\bsk-\alpha)_{-\uu}!}
\widehat f(\bstau_\uu,\bsk_{-\uu})^2\\
&=\sum_{\bsk\in \NN^s}\frac{1}{r_{s,\alpha,\gamma}(\bsk)}\widehat f(\bsk)^2 =\|f\|^2_{r_{s,\alpha,\gamma}}.
\end{align*}
\end{proof}

Equations~\eqref{eq:inprod} and \eqref{eq:norm} show that the Hermite space with the present choice of $r_{s,\alpha,\bsgamma}$ can be interpreted as a Gaussian ANOVA space on the $\RR^s$ or as a Gaussian unanchored Sobolev space of functions on the $\RR^s$.

We have an interesting integral representation of the kernel in the one-dimensional case. In the following we use the notation $\Phi(y):=\int_{-\infty}^y \varphi(\eta)\rd \eta$ and 
\[
\vartheta(x,y):=1_{(-\infty,x]}(y)\Phi(y)-1_{(x,\infty)}(y)\Phi(-y)
\]
for $x,y \in \RR$.

\begin{theorem}\label{pr:intrep}
For $\alpha \in \NN$ and $\gamma>0$ we have
\begin{eqnarray*}
\lefteqn{K_{r_{\alpha,\gamma}}(x,y)}\\
& = & 1+\gamma\sum_{k=1}^{\alpha-1} \frac{H_k(x)H_k(y)}{k!}\\
& & + \gamma \int_{\RR} \frac{1}{\varphi(s)} \left(\int_{\RR^{2\alpha-2}} \vartheta_\alpha(x,\xi_{\alpha-1},\ldots,\xi_1,s)\vartheta_\alpha(y,\eta_{\alpha-1},\ldots,\eta_1,s)  \prod_{k=1}^{\alpha-1}\big(\rd \xi_k \rd\eta_k\big)\right) \rd s,
\end{eqnarray*}
where $\vartheta_n(z_1,\ldots,z_{n+1}):=\prod_{k=1}^{n}\vartheta(z_{k},z_{k+1})$ for $n\in \NN$.
\end{theorem}

A  proof for  this representation will be given in Appendix~\ref{appA}.

\subsection{A first variant of the Gaussian ANOVA space $\cH_{r_{s,\alpha,\bsgamma}}$}

In \cite{IL}, Irrgeher and Leobacher define  in a similar way a variant of the Hermite space. They consider the reproducing kernel  \eqref{def:Hkernel} with corresponding inner product given by  $R(\bsk)=\rho_{s,\alpha,\bsgamma}(\bsk) := \prod_{j=1}^s \rho_{\alpha,\gamma_j}(k_j)$ with \begin{equation*}
\rho_{\alpha,\gamma}(k) :=
\left\{\begin{array}{ll}
1 & \text{for } k=0, \\[0.5em]
\frac{\gamma}{k^{\alpha}}  & \text{for } k \ge 1,
\end{array}\right.
\end{equation*}
for $\alpha \ge 1$ and a generic weight $\gamma \in (0,1]$. Note that these Fourier weights are equal to those of  the classical Korobov space of smoothness $\alpha$. Denote the corresponding reproducing kernel Hilbert space by $\cH_{\rho_{s,\alpha,\bsgamma}}:=\cH(K_{\rho_{s,\alpha,\bsgamma}})$.

For the norm $\|\cdot\|_{\rho_{s,\alpha,\bsgamma}}$ we do not have a representation as a Sobolev type norm like in \eqref{eq:norm} for the norm $\|\cdot\|_{r_{s,\alpha,\bsgamma}}$.

\begin{proposition}\label{pr:r_rho}
We have $$\|f\|_{r_{s,\alpha,\bsgamma}} \le \|f\|_{\rho_{s,\alpha,\bsgamma}} \le \|f\|_{r_{s,\alpha,\bsgamma/\alpha^{\alpha}}},$$ where $\bsgamma/\alpha^{\alpha}:=(\gamma_j/\alpha^{\alpha})_{j \ge 1}$. In particular $\cH_{\rho_{s,\alpha,\bsgamma}}$ is continuously embedded in the space $\cH_{r_{s,\alpha,\bsgamma}}$ and the norm of the embedding operator is bounded by~1. 
\end{proposition}

\begin{proof}
According to Lemma~\ref{le:bdrk} we have 
$$\rho_{\alpha,\gamma}(k) \le r_{\alpha,\gamma}(k) \le \alpha^{\alpha} \rho_{\alpha,\gamma}(k)=\rho_{\alpha,\alpha^{\alpha}\gamma}(k)\quad \mbox{for all $k \in \NN$}$$ and obviously $\rho_{\alpha,\gamma}(0)=r_{\alpha,\gamma}(0)=1$. Hence $$\rho_{s,\alpha,\bsgamma}(\bsk) \le r_{s,\alpha,\bsgamma}(\bsk) \le \alpha^{\alpha |\uu(\bsk)|} \rho_{s,\alpha,\bsgamma}(\bsk) \quad \mbox{for all $\bsk \in \NN_0^s$,}$$ where for $\bsk \in \NN_0^s$ we write $\uu(\bsk):=\{j \in [s]\ : \ k_j \not=0\}$, and hence
\begin{eqnarray*}
\|f\|_{r_{s,\alpha,\bsgamma}}^2 & = & \sum_{\bsk \in \NN_0^s} \frac{1}{r_{s,\alpha,\bsgamma}(\bsk)} |\widehat{f}(\bsk)|^2
 \le  \sum_{\bsk \in \NN_0^s} \frac{1}{\rho_{s,\alpha,\bsgamma}(\bsk)} |\widehat{f}(\bsk)|^2
 =  \|f\|_{\rho_{s,\alpha,\bsgamma}}^2\\
& = & \sum_{\uu \subseteq [s]} \sum_{\bsk_{\uu} \in \NN^{|\uu|}} \frac{1}{\prod_{j \in \uu} \rho_{\alpha,\gamma_j}(k_j)} \, |\widehat{f}(\bsk_{\uu},0)|^2\\
& \le & \sum_{\uu \subseteq [s]} \sum_{\bsk_{\uu} \in \NN^{|\uu|}} \frac{1}{\prod_{j \in \uu} (r_{\alpha,\gamma_j}(k_j)/\alpha^{\alpha})} \, |\widehat{f}(\bsk_{\uu},0)|^2\\
& = & \sum_{\uu \subseteq [s]} \sum_{\bsk_{\uu} \in \NN^{|\uu|}} \frac{1}{\prod_{j \in \uu} r_{\alpha,\gamma_j/\alpha^{\alpha}}(k_j)} \, |\widehat{f}(\bsk_{\uu},0)|^2
 =  \|f\|_{r_{s,\alpha,\bsgamma/\alpha^{\alpha}}}^2.
\end{eqnarray*} 
Here for $\bsk=(k_1,\ldots,k_s)$ and $\uu \subseteq [s]$ we write $(\bsk_{\uu},0)$ for the $s$-dimensional vector whose $j$-th component is $k_j$ if $j \in \uu$ and 0 otherwise.
\end{proof}

\subsection{A second variant of the Gaussian ANOVA space $\cH_{r_{s,\alpha,\bsgamma}}$}
 
In \cite{DILP18} a further Sobolev type norm was considered, namely
\begin{align*}
\|f\|^2 &:=\sum_{\bstau \in\{0,\ldots,\alpha\}^s} \int_{\RR^s} (\partial_{\bsx}^{\bstau} f(\bsx))^2 \varphi_s(\bsx)\rd\bsx.
\end{align*}
A weighted variant of this is
\begin{align}\label{def:normv2}
\|f\|_{\psi_{s,\alpha,\bsgamma}}^2 &:=\sum_{\bstau \in\{0,\ldots,\alpha\}^s} \Bigg(\prod_{j=1\atop \tau_j \not=0}^s\gamma_j^{-1}\Bigg) \int_{\RR^s} (\partial_{\bsx}^{\bstau} f(\bsx))^2 \varphi_s(\bsx)\rd\bsx.
\end{align}
The meaning of $\psi$ will be explained shortly.


\begin{example}\label{ex2}\rm
As an example, in the univariate case with a generic weight $\gamma>0$ the squared norm can be written in the form $$\|f\|_{\psi_{1,\alpha,\bsgamma}}^2 = \int_{\RR} (f(x))^2 \varphi(x) \rd x + \frac{1}{\gamma} \sum_{\tau=1}^{\alpha} \int_{\RR} (f^{(\tau)}(x))^2 \varphi(x) \rd x,$$ which should be compared with \eqref{fo_norm1a} in Example~\ref{ex1}.  
\end{example}

Likewise, the norm \eqref{def:normv2} can be represented as a Hermite-type norm and this will explain the $\psi$ in our notation. Using \eqref{eq:partabl} we have $$\int_{\RR^s} (\partial_{\bsx}^{\bstau} f(\bsx))^2 \varphi_s(\bsx)\rd\bsx = \sum_{\bsk \ge \bstau} \frac{\bsk!}{(\bsk-\bstau)!}\, (\widehat{f}(\bsk))^2.$$ Hence
\begin{eqnarray*}
\|f\|_{\psi_{s,\alpha,\bsgamma}}^2 & = & \sum_{\bstau \in\{0,\ldots,\alpha\}^s} \Bigg(\prod_{j=1\atop \tau_j \not=0}^s\gamma_j^{-1}\Bigg) \sum_{\bsk \ge \bstau} \frac{\bsk!}{(\bsk-\bstau)!}\, (\widehat{f}(\bsk))^2\\
& = & \sum_{\bsk \in \NN_0^s} \left(\sum_{\bstau \in\{0,\ldots,\alpha\}^s \atop \bstau \le \bsk} \Bigg(\prod_{j=1\atop \tau_j \not=0}^s\gamma_j^{-1}\Bigg)\frac{\bsk!}{(\bsk-\bstau)!}\right) \, (\widehat{f}(\bsk))^2\\
& = & \sum_{\bsk \in \NN_0^s} \left(\prod_{j=1}^s\left(1+\frac{1}{\gamma_j} \sum_{\tau=1 \atop \tau \le k_j}^{\alpha} \frac{k_j!}{(k_j-\tau)!}\right)\right)  \, (\widehat{f}(\bsk))^2\\
& = & \sum_{\bsk \in \NN_0^s} \left(\prod_{j=1}^s\left(1+\frac{1}{\gamma_j} \sum_{\tau=1}^{\alpha} \beta_{\tau}(k_j)\right)\right)  \, (\widehat{f}(\bsk))^2,
\end{eqnarray*}
where for $k \in \NN_0$, 
$$\beta_{\tau}(k):=\left\{ 
\begin{array}{ll}
\frac{k!}{(k-\tau)!} & \mbox{if $k \ge \tau$,}\\[0.5em]
0 & \mbox{otherwise.}
\end{array}
\right.$$

Setting, for $k \in \NN_0$ and a generic weight $\gamma>0$, $$\psi_{\alpha,\gamma}(k):=\left(1+\frac{1}{\gamma} \sum_{\tau=1}^{\alpha} \beta_{\tau}(k_j)\right)^{-1}$$ and for $\bsk=(k_1,\ldots,k_s)\in \NN_0$, $\psi_{s,\alpha,\bsgamma}(\bsk):=\prod_{j=1}^s \psi_{\alpha,\gamma_j}(k_j)$, then 
$$\|f\|_{\psi_{s,\alpha,\bsgamma}}^2 = \sum_{\bsk \in \NN_0^s} \frac{1}{\psi_{s,\alpha,\bsgamma}(\bsk)}  \, (\widehat{f}(\bsk))^2.$$ 
Thus, via the norm $\|\cdot\|_{\psi_{s,\alpha,\bsgamma}}$  we obtain a Hermite space  $\cH_{\psi_{s,\alpha,\bsgamma}}$ with reproducing kernel of the form \eqref{def:Hkernel} with Fourier weights $R(\bsk)= \psi_{s,\alpha,\bsgamma}(\bsk)$. 

\begin{remark}\rm
Using the method of Thomas-Agnan~\cite{thomas96} the kernel $K_{\psi_{1,1,\gamma}}$ ($s=1$ and $\alpha=1$) can be expressed by means of solutions of the second order differential equation 
\begin{equation*}
g''(y)=y g'(y)+\gamma g(y) 
\end{equation*}
with certain boundary conditions. We omit the details of this observation.
\end{remark}

\begin{proposition}\label{pr:r_psi}
We have $$\|f\|_{r_{s,\alpha,\bsgamma}}\le \|f\|_{\psi_{s,\alpha,\bsgamma}}\le \|f\|_{r_{s,\alpha,\bsgamma/(2 \alpha^{\alpha})}},$$ where $\bsgamma/(2 \alpha^{\alpha}):=(\gamma_j/(2 \alpha^{\alpha}))_{j \ge 1}$. In particular, $\cH_{\psi_{s,\alpha,\bsgamma}}$ is continuously embedded in $\cH_{r_{s,\alpha,\bsgamma}}$ and the norm of the embedding operator is bounded by~1. 
\end{proposition}

\begin{proof}
Using the Cauchy-Schwarz inequality we obtain
\begin{align*}
\lefteqn{\int_{\RR^{s-|\uu|}}\left(\int_{\RR^{|\uu|}}\partial_\bsx^{(\bstau_\uu,\alpha_{-\uu})} f(\bsx)\varphi_{|\uu|}(\bsx_\uu)\rd\bsx_\uu\right)^2\varphi_{s-|\uu|}(\bsx_{-\uu})\rd\bsx_{-\uu}}\\
&\le \int_{\RR^{s-|\uu|}}\int_{\RR^{|\uu|}}\left(\partial_\bsx^{(\bstau_\uu,\alpha_{-\uu})} f(\bsx)\right)^2\varphi_{|\uu|}(\bsx_\uu)\rd\bsx_\uu\,\varphi_{s-|\uu|}(\bsx_{-\uu})\rd\bsx_{-\uu}\\
&= \int_{\RR^{s}}\left(\partial_\bsx^{(\bstau_\uu,\alpha_{-\uu})} f(\bsx)\right)^2\varphi_{s}(\bsx)\rd\bsx,
\end{align*}
such that  $\|f\|_{r_{s,\alpha,\bsgamma}}\le \|f\|_{\psi_{s,\alpha,\bsgamma}}$ for all $f\in \cH_{\psi_{s,\alpha,\bsgamma}}$. 

On the other hand, for $k \in \NN$ we have $$\sum_{\tau=0}^{\alpha} \beta_{\tau}(k) \le 2 \, k^{\alpha},$$ because:
\begin{itemize}
\item If $k > \alpha$, then $$\sum_{\tau=0}^{\alpha} \beta_{\tau}(k) =\sum_{\tau=0}^{\alpha} \frac{k!}{(k-\tau)!} \le \sum_{\tau=0}^{\alpha} k^{\tau}=\frac{k^{\alpha+1}-1}{k-1} \le 2\, k^{\alpha}.$$
\item If $1\le k \le \alpha$, then  $$\sum_{\tau=0}^{\alpha} \beta_{\tau}(k) =\sum_{\tau=0}^k \frac{k!}{(k-\tau)!} = k! \sum_{\tau=0}^k \frac{1}{\tau!} \le 2\, k^k \le 2\, k^{\alpha}.$$
\end{itemize} 
Therefore and with Lemma~\ref{le:bdrk}, for $k \in \NN$ we obtain $$\psi_{\alpha,\gamma}(k) \ge \frac{\gamma}{\sum_{\tau=0}^{\alpha} \beta_{\tau}(k)} \ge \frac{\gamma}{2\, k^{\alpha}} \ge \frac{1}{2\, \alpha^{\alpha}} \, r_{\alpha,\gamma}(k).$$ Again, $\psi_{\alpha,\gamma}(0)=1=r_{\alpha,\gamma}(0)$. Hence $$\psi_{s,\alpha,\bsgamma}(\bsk) \ge \left(\frac{1}{2\, \alpha^{\alpha}}\right)^{|\uu(\bsk)|}\, r_{s,\alpha,\bsgamma}(\bsk) \quad \mbox{for all $\bsk \in \NN_0^s$.}$$ Like in the proof of Proposition~\ref{pr:r_rho}, this implies that 
\begin{eqnarray*}
\|f\|_{\psi_{s,\alpha,\bsgamma}} \le \|f\|_{r_{s,\alpha,\bsgamma/(2 \alpha^{\alpha})}}.
\end{eqnarray*}
\end{proof}

\subsection{An anchored space of Sobolev type} \label{subsec:anch}
 
For the sake of completeness we mention also an anchored variant of the ANOVA norm \eqref{eq:norm} with anchor $\bszero=(0,\ldots,0)$, which is given by
$$
\|f\|^2_{\pitchfork,s,\alpha,\bsgamma}:=\sum_{\uu\subseteq[s]}\sum_{\bstau_{\uu}\in\{0,\ldots,\alpha-1\}^{|\uu|}}\gamma_{\bstau_{\uu}}^{-1} \int_{\RR^{s-|\uu|}}\left(\partial_{\bsx}^{(\bstau_{\uu},\alpha_{-\uu})}f(\bsx_{- \uu},0)\right)^2\varphi_{s-|\uu|}(\bsx_{-\uu})\rd\bsx_{-\uu}.
$$

\begin{example}\label{ex:anch1}\rm
For $s=1$ and a generic weight $\gamma>0$  we have
 \[
\|f\|^2_{\pitchfork,1,\alpha,\gamma}=(f(0))^2+\frac{1}{\gamma}\sum_{k=1}^{\alpha-1} (f^{(k)}(0))^2 + \frac{1}{\gamma} \int_\RR (f^{(\alpha)}(y))^2\varphi(y)\rd y.
\]
\end{example}

Denote the corresponding function space by $\cH_{\pitchfork,s,\alpha,\bsgamma}$. Also this space,  the so-called anchored space is a reproducing kernel Hilbert space of tensor product form. For $\bsx,\bsy \in \RR^s$ the reproducing kernel is $$K_{\pitchfork,s,\alpha,\bsgamma}(\bsx,\bsy):=\prod_{j=1}^s K_{\pitchfork,\alpha,\gamma_j}(x_j,y_j),$$ where the kernel $K_{\pitchfork,\alpha,\gamma}$ in the case $s=1$ ist given in the following proposition.

\begin{proposition}\label{pr:rep_anch}
For $x,y \in \RR$ and a generic weight $\gamma>0$ we have
\begin{eqnarray}
K_{\pitchfork,\alpha,\gamma}(x,y) & = & 1+\gamma \sum_{\ell=1}^{\alpha-1} \frac{(x y)^{\ell}}{(\ell!)^2}\label{eq:anchored-kernel-finitea}\\
&& + \gamma \, 1_{[0, \infty)}(x \, y) \int_0^{\infty} \frac{1}{\varphi(s)} \, \frac{(|x|-s)_+^{\alpha-1} (|y|-s)_+^{\alpha-1}}{((\alpha-1)!)^2} \rd s.\nonumber
\end{eqnarray}
\end{proposition}

We omit the proof of this formula. For a similar space and kernel  we refer to \cite[Sec.~12.5.1]{NW10}. 

We see from Proposition~\ref{pr:rep_anch} that the part

\begin{equation}\label{prop:decomposable}
L(x,y):= 1_{[0, \infty)}(x \, y) \int_0^{\infty} \frac{1}{\varphi(s)} \, \frac{(|x|-s)_+^{\alpha-1}(|y|-s)_+^{\alpha-1}}{((\alpha-1)!)^2} \rd s
\end{equation}
of the kernel $K_{\pitchfork,\alpha,\gamma}$ is decomposable at 0, meaning that $L(x,y)=0$ whenever $x < 0 < y$ or $y < 0 < x$.

\begin{remark}
In general the anchored space $\cH_{\pitchfork,s,\alpha,\bsgamma}$ is not a Hermite space in the sense of the definition in Section~\ref{intro}. To see this, write (for $s=1$) $\cH_{\pitchfork,\alpha,\bsgamma}=\cH_1+\cH_2$,  where $\cH_1$ is the closed subspace of all polynomials of degree smaller than $\alpha$ and $\cH_2$ is the orthogonal complement of $\cH_1$ in $\cH_{\pitchfork,\alpha,\bsgamma}$.

Using property (7) from \cite[Section~2]{aronszajn50}, we see that $K_{\pitchfork,\alpha,\gamma}=K_1+K_2$, where $K_j$ is a reproducing kernel for $\cH_j$, $j \in \{1,2\}$. Using the representation of $K_{\pitchfork,\alpha,\gamma}$ from Proposition~\ref{pr:rep_anch} gives $K_1(x,y)=1+\gamma \sum_{\ell=1}^{\alpha-1} \frac{(x y)^{\ell}}{(\ell!)^2}$.

Now assume, in order to reach a contradiction, that $\cH_{\pitchfork,3,1}$ {\em is} a Hermite space, and therefore
there exists $R$ with $K_{\pitchfork,3,1}(x,y)=\sum_{k=0}^\infty R(k)H_k(x)H_k(y)$. But then 
$K_1(x,y)=\sum_{k=0}^2 R(k)H_k(x)H_k(y)$, so
\begin{align*}
1+xy+\frac{(xy)^2}{2}
&=K_{1}(x,y)=1+R(1)xy+\frac{1}{2}R(2)(x^2-1)(y^2-1)\\
&=1+\frac{1}{2}R(2)+R(1)xy+\frac{1}{2}R(2)x^2y^2-\frac{1}{2}R(2)x^2-\frac{1}{2}R(2)y^2
\end{align*} 
But now comparing coefficients yields $R(2)=0$ and $R(2)=1$, the desired contradiction.
\end{remark}

\begin{remark}
It is worth noting that, while $\cH_{\pitchfork,s,\alpha,\bsgamma}$ and $\cH_{r_{s,\alpha,\bsgamma}}$ are certainly 
equivalent as Banach spaces,  in general the norm of neither space is dominated by that of the other. 
To see this, let $s=1$, $\alpha=3$, $\gamma=1$,  $a,b\in \RR$ and consider the function $f\colon \RR\to\RR$ with 
$f(x)=a+\frac{b}{2}x^2$. Then $f^{(3)}\equiv 0$, so that $\int_\RR |f^{(3)}(y)|^2\varphi(y)\rd y=0$.
Now
\begin{align*}
\|f\|^2_{\pitchfork,1,3,1}&=(f(0))^2+ (f'(0))^2+ (f''(0))^2=a^2+b^2\\
\|f\|_{r_{1,3,1}}^2 &=\left(\int_\RR \left(a+\frac{b}{2}y^2\right)\varphi(y)\rd y\right)^2+\left(\int_\RR by \varphi(y)\rd y\right)^2+\left(\int_\RR b\varphi(y)\rd y\right)^2\\
&=\left(a+\frac{b}{2}\right)^2+b^2\,.
\end{align*}
Thus, by choosing $a=1$, $b=2$ we get 
\(\|f\|^2_{\pitchfork,1,3,1}=5<8=\|f\|_{r_{1,3,1}}^2\), while by choosing $a=1$, $b=-2$ we find 
\(\|f\|^2_{\pitchfork,1,3,1}=5>4=\|f\|_{r_{1,3,1}}^2\)\,.
\end{remark}

\section{Integration and $L_2$-approximation in Hermite spaces}\label{sec:int_app_general}

We consider integration and $L_2$-approximation for functions from a weighted Hermite space $H_R$ where our main focus will be on $R=r_{s,\alpha,\bsgamma}$. Throughout we assume that $R(\bszero)=1$ and $0< R(\bsh) \le 1$ for all $\bsh \in \NN_0^s$.

\paragraph{The integration problem.} The multivariate integration problem is given by $\INT_R:\cH_R \rightarrow \RR$, 
\begin{equation*}
\INT_R(f) = \int_{\RR^s} f(\bsx) \, \varphi_{s}(\bsx) \rd \bsx = \widehat{f}(\bszero) \quad \text{for} \quad f \in \cH_R. 
\end{equation*}
In order to approximate $\INT_R$ we use linear algorithms of the form 
\begin{equation}\label{def:alglin:int}
A_{n,s}^{{\rm int}}(f):= \sum_{i=1}^n w_i f(\bsx_i)
\end{equation}
with  nodes $\bsx_1,\ldots,\bsx_n \in \RR^s$ and integration weights $w_1,\ldots,w_s \in \RR$. The quality of the algorithm is measured in terms of the worst-case error which is defined by 
$$e^{{\rm int}}(A_{n,s}^{{\rm int}}):=\sup_{\substack{f \in \cH_R \\ \|f\|_R \le 1}}\left|\INT_R(f) -A_{n,s}^{{\rm int}}(f)\right|.$$ The $n$-th minimal error for integration in $\cH_R$ is defined as 
\begin{equation*}
e(n,\INT_R):=\inf_{A_{n,s}^{{\rm int}}} e^{{\rm int}}(A_{n,s}^{{\rm int}}),
\end{equation*}
where the infimum is extended over all linear algorithms of the form \eqref{def:alglin:int} using $n$ function evaluations and integration weights, respectively.

The initial (integration) error is $e^{{\rm int}}(0,\INT_R)=\|\INT_R\|=1$, because
\begin{align*}
	\|\INT_R\|
	=
	\sup_{0 \ne f \in \cH_R} \frac{|\widehat{f}(\bszero)|}{\|f\|_R}
	=
	\sup_{0 \ne f \in \cH_R} \frac{|\widehat{f}(\bszero)|}{\sqrt{\sum_{\bsh \in \NN_0^s} R^{-1}(\bsh) \, |\widehat{f}(\bsh)|^2}}
	\le
	\sup_{0 \ne f \in \cH_R} \frac{|\widehat{f}(\bszero)|}{\sqrt{ |\widehat{f}(\bszero)|^2}}	
	=
	1
\end{align*}
and for $g=1 \in \cH_R$ we have that
\begin{equation*}
\frac{|\widehat{g}(\bszero)|}{\|g\|_R} = \frac{\int_{\RR^s} \varphi_{s}(\bsx) \rd \bsx}{\int_{\RR^s} \varphi_{s}(\bsx) \rd \bsx} = 1.
\end{equation*}

\paragraph{The $L_2$-approximation problem.} The $L_2$-approximation of functions from the Hermite space $\cH_R$ is given by the embedding operator $\APP_R: \cH_R \to L_2(\RR^s, \varphi_s)$ with $$\APP_R (f) = f \quad \mbox{ for }\quad f \in \cH_R.$$  In order to approximate $\APP_R$ with respect to the norm $\|\cdot\|_{L_2}$ we will employ linear algorithms $A_{n,s}^{{\rm app}}$ that use $n$ information evaluations 
and are of the form
\begin{equation} \label{eq:alg_form}
	A_{n,s}^{{\rm app}}(f)
	=
	\sum_{i=1}^n L_i(f) \, g_i
	\quad \text{for }
	f \in \cH_R	
\end{equation}
with functions $g_i \in L_2(\RR^s, \varphi_s)$ and bounded linear functionals $L_i \in \cH^\ast_R$ for $i \in \{1,2,\ldots,n\}$ (see \cite[Theorem~4.8]{NW08} or \cite{TWW}).  If,
for an algorithm $A_{n,s}^{{\rm app}}$ as  in \eqref{eq:alg_form} all $L_i$ are from the same information class 
$\Lambda  \subseteq \cH^\ast_R$, then we simply write with some abuse of notation $A_{n,s}^{{\rm app}}\in \Lambda$.

In this paper we consider two classes of permissible information, namely the class $ \Lambda^{{\rm all}}$ consisting of all continuous linear functionals, i.e., $\Lambda^{{\rm all}}=\cH^\ast_R$, and the class $\Lambda^{{\rm std}}$ consisting exclusively of point evaluation functionals. Since $\cH_R$ is a reproducing kernel Hilbert space it is clear that point evaluation functionals are continuous and hence $\Lambda^{{\rm std}}  \subseteq \Lambda^{{\rm all}}$.

We remark that the embedding operator $\APP_R$ is continuous for all $s \in \NN$, which can be seen as follows.
We have for all $f \in \cH_R$ that
\begin{align*}
\|\APP_R(f)\|_{L_2}^2
	&=
	\|f\|_{L_2}^2 
	= \sum_{\bsh \in \NN_0^s} |\widehat{f}(\bsh)|^2
	\le 
	\sum_{\bsh \in \NN_0^s} \frac{1}{R(\bsh)} |\widehat{f}(\bsh)|^2 
	= 
	\|f\|_{R}^2 < \infty
	,
	\end{align*}
	where we used Parseval's identity and the fact that $0 <  R(\bsh) \le 1$ for all $\bsh \in  \NN_0^s$. By considering the choice $f \equiv 1$, it follows that the above inequality is sharp, such that the operator norm of $\APP_R$ is given by
	\begin{equation*}
	\|\APP_R\| =1.
	\end{equation*}

\begin{remark}\rm
Note that it does not make sense to study $L_\infty$-approximation for the Hermite space $\cH_R$ since this problem is not well defined because $$K_R(\bsx,\bsx)=\sum_{\bsh \in \NN_0^s} R(\bsh) (H_{\bsk}(\bsx))^2 \ge 1+R(\bsone) x_1^2\cdots x_s^2$$ and hence (see \cite[Section~2]{KWW08}) $$\sup_{f \in \cH_R \atop \|f\|_R\le 1}\|f\|_{L_{\infty}} = \esssup_{\bsx \in \RR^s}  \sup_{f \in \cH_R \atop \|f\|_R\le 1}|f(\bsx)| =  \esssup_{\bsx \in \RR^s} \sqrt{K_R(\bsx,\bsx)}=\infty.$$
\end{remark}

The worst-case error of an algorithm $A_{n,s}^{{\rm app}}$ of the form   \eqref{eq:alg_form} is defined by
\begin{equation*}
	e^{{\rm app}}(A_{n,s}^{{\rm app}})
	:=
	\sup_{\substack{f \in \cH_R \\ \|f\|_R \le 1}} \|\APP_R (f) - A_{n,s}^{{\rm app}}(f)\|_{L_2(\RR^s,\varphi_s)}
\end{equation*}
and the $n$-th minimal worst-case error w.r.t.~the information class $\Lambda$ is given by
\begin{equation*}
	e(n,\APP_R;\Lambda)
	:=
	\inf_{A_{n,s}^{{\rm app}} \in \Lambda} e^{{\rm app}}(A_{n,s}^{{\rm app}})
	.
\end{equation*}
Since $\Lambda^{{\rm std}}  \subseteq \Lambda^{{\rm all}}$ it follows that 
\begin{equation}\label{est:err_allstd}
e(n,\APP_R;\Lambda^{{\rm all}}) \le e(n,\APP_R;\Lambda^{{\rm std}}).
\end{equation}

\paragraph{A relation between integration and $L_2$-approximation.}

\begin{proposition}\label{pr:vglintapperr}
For the space $\cH_R$ we have
$$e(n,\INT_R) \le e(n,\APP_R; \Lambda^{{\rm std}}).$$
\end{proposition}

\begin{proof}
Recall that $\|\INT_R\|=1$. Using Parseval's identity, we have for any algorithm of the form $A_{n,s}^{{\rm app}}(f)=\sum_{i=1}^n g_i f(\bsx_i)$ with $\bsx_i \in \RR^s$ and $g_i \in L_2(\RR^s, \varphi_s)$ for $i \in \{1,2,\ldots,n\}$ that
\begin{equation*}
\|\APP_R (f) - A_{n,s}^{{\rm app}}(f)\|_{L_2(\RR^s,\varphi_s)}^2 = \sum_{\bsk\in\NN_0^s} \left| \widehat{f}(\bsk) - \widehat{A_{n,s}^{{\rm app}}(f)}(\bsk) \right|^2,
\end{equation*}
where $\widehat{A_{n,s}^{{\rm app}}(f)}(\bsk)$ is the $\bsk$-th Hermite coefficient given by
\begin{equation*}
\widehat{A_{n,s}^{{\rm app}}(f)}(\bsk) = \sum_{i=1}^n f(\bsx_i) \int_{\RR^s} g_i(\bsx) H_{\bsk}(\bsx) \varphi_s(\bsx) \rd \bsx.
\end{equation*}
This gives
\begin{align*}
\|\APP_R (f) - A_{n,s}^{{\rm app}}(f)\|_{L_2(\RR^s,\varphi_s)}^2
&=\sum_{\bsk\in\NN_0^s} \left| \widehat{f}(\bsk) - \widehat{A_{n,s}^{{\rm app}}(f)}(\bsk) \right|^2\\
&\ge \left| \widehat{f}(\bszero) - \widehat{A_{n,s}^{{\rm app}}(f)}(\bszero) \right|^2 \\
&= \left| \INT_R(f) - \sum_{i=1}^n f(\bsx_i) \int_{\RR^s} g_i(\bsx) \varphi_s(\bsx) \rd \bsx \right|^2\\
&=\left| \INT_R(f)  - \sum_{i=1}^n w_i f(\bsx_i) \right|^2\\
& = \left| \INT_R(f)  - A_{n,s}^{{\rm int}}(f) \right|^2
\end{align*}
where $$w_i:= \int_{\RR^s} g_i(\bsx) \varphi_s(\bsx) \rd \bsx \quad\mbox{for $i \in \{1,2,\ldots,n\}$}$$ and $$A_{n,s}^{{\rm int}}(f):= \sum_{i=1}^n w_i f(\bsx_i).$$ Thus, for every linear approximation algorithm $A_{n,s}^{{\rm app}}$ we can find a linear integration algorithm $A_{n,s}^{{\rm int}}$ such that $$e^{{\rm int}}(A_{n,s}^{{\rm int}}) \le e^{{\rm app}}(A_{n,s}^{{\rm app}}).$$ From this we conclude that $$e(n,\INT_R) \le e(n,\APP_R; \Lambda^{{\rm std}}).$$ 
\end{proof}

The next proposition provides some relations between worst-case errors for different but related Hermite spaces. 

\begin{proposition}\label{pr:embb:err}
Let $R_1,R_2: \NN_0^s \rightarrow \RR$ be two Fourier weights for Hermite spaces $\cH_{R_1}$ and $\cH_{R_2}$ such that for the corresponding norms we have $$\|f\|_{R_1} \le \|f\|_{R_2} \quad \mbox{for all $f \in \cH_{R_2}$.}$$ Then for all $n \in \NN$ we have 
$$e(n,\INT_{R_2}) \le e(n,\INT_{R_1})$$ and $$e(n,\APP_{R_2};\Lambda) \le e(n,\APP_{R_1};\Lambda)\quad \mbox{for $\Lambda \in \{\Lambda^{{\rm all}},\Lambda^{{\rm std}}\}$.}$$
\end{proposition}

We omit the easy proof of this result and refer to \cite[Proposition~7.5]{DKP22}.

\begin{remark}
Under our assumption that $R_1$ and $R_2$ vanish nowhere on $\NN_0^s$,
it is not hard to check that, 
\[
\|f\|_{R_1} \le \|f\|_{R_2} \Leftrightarrow \big[R_1(\bsk)\ge R_2(\bsk) \text{ for all } \bsk\in \NN_0^s\big].
\]
\end{remark}
\paragraph{Tractability.} We are interested in how the worst-case errors of algorithms $A_{n,s}^{\bullet}$, $\bullet \in \{{\rm int},{\rm app}\}$, depend on the number $n$ of information evaluations used and on the problem dimension~$s$. To this end, we define the so-called information complexity as
\begin{equation*}
	n(\varepsilon,S;\Lambda)
	:=
	\min\{n \in \NN_0 \ : \  e(n,S;\Lambda) \le \varepsilon \}
\end{equation*}
where $S \in \{\INT_R,\APP_R\}$, with $\varepsilon \in (0,1)$ and $s \in \NN$. Note that here we do not need to distinguish between the absolute and the normalized error criterion since in the present case the related initial errors equal 1. If $S=\INT_R$ it obviously makes only sense to consider the class $\Lambda^{{\rm std}}$ and hence we just write $n(\varepsilon,\INT_R)$. 

Obviously, \eqref{est:err_allstd} implies that 
\begin{equation}\label{est:icomp_allstd}
n(\varepsilon,\APP_R;\Lambda^{{\rm all}}) \le n(\varepsilon,\APP_R;\Lambda^{{\rm std}})
\end{equation}
and Proposition~\ref{pr:vglintapperr} implies 
\begin{equation}\label{vgl:n:int:app}
n(\varepsilon, \INT_R) \le n(\varepsilon,\APP_R;\Lambda^{{\rm std}}).
\end{equation}

 In order to characterize the dependency of the information complexity on the dimension $s$ and the error threshold $\varepsilon$, we will study several notions of tractability which are given in the following definition.

\begin{definition}
Let $S \in \{\INT_R,\APP_R\}$.  We say we have:
	\begin{enumerate}[label=\rm{(\alph*)}]
		\item Polynomial tractability (PT) if there exist non-negative numbers $\tau, \sigma, C$ such that
		\begin{equation*}
			n(\varepsilon,S;\Lambda)
			\le
			C \, \varepsilon^{-\tau} s^\sigma
			\quad \text{for all} \quad
			s \in \NN, \varepsilon \in (0,1)
			.
		\end{equation*}
		\item Strong polynomial tractability (SPT) if there exist non-negative numbers $\tau, C$ such that
		\begin{equation*}
			n(\varepsilon,S;\Lambda)
			\le
			C \, \varepsilon^{-\tau}
			\quad \text{for all} \quad
			s \in \NN, \varepsilon \in (0,1)
			.
		\end{equation*}
		In that case we define the {\em exponent of SPT} as
		\begin{equation*}
		\inf\{\tau \colon \exists C>0 \text{ such that } n(\varepsilon,S;\Lambda)
			\le
			C \, \varepsilon^{-\tau}\ 
			 \forall  
			s \in \NN, \varepsilon \in (0,1)\}.
		\end{equation*}
		\item Weak tractability (WT) if
		\begin{equation*}
			\lim_{s + \varepsilon^{-1}\rightarrow \infty} \frac{\log n(\varepsilon,S;\Lambda)}{s + \varepsilon^{-1}}
			=
			0
			.
		\end{equation*}
		\item Quasi-polynomial tractability (QPT) if there exist non-negative numbers $\tau, C$ such that
		\begin{equation*}
			n(\varepsilon,S;\Lambda)
			\le
			C \, \exp(\tau (1 + \log s) (1 + \log \varepsilon^{-1}))
			\quad \text{for all} \quad
			s \in \NN, \varepsilon \in (0,1)
			.
		\end{equation*} 
		In that case we define the {\em exponent of QPT} as
		\begin{equation*}
		\inf\{\tau \colon \exists C>0 \text{ such that } n(\varepsilon,S;\Lambda)
			\le
			C \, \exp(\tau (1 + \log s) (1 + \log \varepsilon^{-1})) \
			 \forall  
			s \in \NN, \varepsilon \in (0,1)\}.
		\end{equation*}
		\item $(\sigma,\tau)$-weak tractability ($(\sigma,\tau)$-WT) if there exist positive $\sigma,\tau$ such that
		\begin{equation*}
			\lim_{s + \varepsilon^{-1}} \frac{\log n(\varepsilon,S;\Lambda)}{s^\sigma + \varepsilon^{-\tau}}
			=
			0
			.
		\end{equation*}
		\item Uniform weak tractability (UWT) if $(\sigma,\tau)$-weak tractability holds for all $\sigma,\tau \in (0,1]$.    
	\end{enumerate} 
\end{definition}

\section{$L_2$-approximation in weighted Hermite spaces}\label{sec:app}

In this section we present results about tractability of $L_2$-approximation for Hermite spaces $\cH_R$ with Fourier weights $R \in \{r_{s,\alpha,\bsgamma},\rho_{s,\alpha,\bsgamma},\psi_{s,\alpha,\bsgamma}\}$. From our examples in Section~\ref{Hspaces} we mainly concentrate on the most comprising weighted Gaussian ANOVA space $\cH_{r_{s,\alpha,\bsgamma}}$ from Section~\ref{sec:ANOVA}. Via embedding we then can derive corresponding results also for the other cases. Throughout we consider the smoothness parameter $\alpha$ and the weights $\bsgamma$ as fixed. With this in mind, we often simplify the notation by just writing $\APP_s$ instead of $\APP_R$ or $\APP_{r_{s,\alpha,\bsgamma}}$.

\subsection{Tractability for the class $\Lambda^{\text{all}}$}

In order to characterize tractability properties of the approximation problem we introduce the following figures: For a weight sequence $\bsgamma=(\gamma_j)_{j \ge 1}$ we will use the infimum $$\bsgamma_I:=\inf_{j \ge 1} \gamma_j$$ and  the so-called sum exponent 
\begin{equation} \label{def_sgamma} 
s_{\bsgamma}:=\inf\left\{\kappa>0 \ : \ \sum_{j=1}^{\infty} \gamma_j^{\kappa} < \infty \right\},
\end{equation}
with the convention that $\inf \emptyset:=\infty$.

First we state the exact ``if and only if'' characterization for tractability of $L_2$-approximation in $\cH_{r_{s,\alpha,\bsgamma}}$ for the information class $\Lambda^{{\rm all}}$.

\begin{theorem} \label{thm:main}
Let $\alpha\ge 1$ and $\bsgamma$ be a sequence of weights. Consider the $L_2$-approximation problem $\APP=(\APP_s)_{s \ge 1}$ for the weighted Hermite spaces $\cH_{r_{s,\alpha,\bsgamma}}$ for $s \in \NN$ and for the information class $\Lambda^{{\rm all}}$. Then we have the following exact characterizations of tractability:
\begin{enumerate}
\item SPT  holds if and only if $s_{\bsgamma}< \infty$. In this case the exponent of SPT is 
\begin{equation*}
\tau^{\ast}(\Lambda^{\mathrm{all}})	= 2 \max\left(\frac{1}{\alpha},s_{\bsgamma}\right) .
\end{equation*}
\item SPT and PT  are equivalent.
\item QPT, UWT and WT are equivalent and hold if and only if $\bsgamma_I <1$. In this case the exponent of QPT is 

\begin{equation*}
t^{\ast}(\Lambda^{{\rm all}}) = 
\begin{cases}
2 \max\left(\frac{1}{\alpha} , \frac{1}{\ln \bsgamma_I^{-1}}\right) & \text{ if }\bsgamma_I \ne 0,\\
2\frac{1}{\alpha}& \text{ if }\bsgamma_I = 0.
\end{cases}
\end{equation*}

\item For $\sigma>1$, $(\sigma,\tau)$-WT holds for all weights $1 \ge \gamma_1 \ge \gamma_2 \ge \ldots > 0$. 
\end{enumerate}
\end{theorem}

From this theorem we can derive the following consequences.

\begin{corollary}\label{cor:app:R:std}
Let $\alpha\ge 1$ and $\bsgamma$ be a sequence of weights. Consider the $L_2$-approximation problem $\APP=(\APP_s)_{s \ge 1}$ for the weighted Hermite spaces $\cH_{R}$, $R \in \{\rho_{s,\alpha,\bsgamma},\psi_{s,\alpha,\bsgamma}\}$ for $s \in \NN$ and for the information class $\Lambda^{{\rm all}}$. Then we have:
\begin{enumerate}
\item SPT  holds if and only if $s_{\bsgamma}< \infty$. In this case the exponent of SPT is 
\begin{equation*}
\tau^{\ast}(\Lambda^{\mathrm{all}})	= 2 \max\left(\frac{1}{\alpha},s_{\bsgamma}\right) .
\end{equation*}
\item SPT and PT  are equivalent.
\item If $\bsgamma_I< \infty$, then we have QPT. In this case the exponent of QPT is 

\begin{equation*}
t^{\ast}(\Lambda^{{\rm all}}) = 
\begin{cases}
2 \max\left(\frac{1}{\alpha} , \frac{1}{\ln \bsgamma_I^{-1}}\right) & \text{ if }\bsgamma_I \ne 0,\\
2\frac{1}{\alpha}& \text{ if }\bsgamma_I \ne 0.
\end{cases}
\end{equation*}

\item For $\sigma>1$, $(\sigma,\tau)$-WT holds for all weights $1 \ge \gamma_1 \ge \gamma_2 \ge \ldots > 0$. 
\end{enumerate}
\end{corollary}

We start with some preparation for the proof of Theorem~\ref{thm:main}.

It is commonly known that the $n$-th minimal worst-case errors $e(n,\APP_s;\Lambda^{{\rm all}})$ are directly related to the eigenvalues of the self-adjoint operator 
\begin{equation}\label{def:Ws}
	W_s 
	:= 
	\APP_s^\ast \APP_s: \cH_{r_{s,\alpha,\bsgamma}} \to \cH_{r_{s,\alpha,\bsgamma}}
	.
\end{equation}
Denote these eigenvalues by $1=\lambda_{s,1} \ge \lambda_{s,2}\ge \lambda_{s,3}\ge \ldots$. Then we have (see \cite[Corollary~4.12]{NW08}) that 
\begin{equation}\label{errEW}
e(n,\APP_s;\Lambda^{{\rm all}})=\lambda_{s ,n+1}^{1/2}
\end{equation}
and hence 
\begin{equation}\label{infcomplam}
n(\varepsilon,\APP_s;\Lambda^{{\rm all}}) =\min\{n\ : \ \lambda_{s ,n+1} \le \varepsilon^2\}.
\end{equation}

In the following lemma, we derive the eigenpairs of the operator $W_s$. For this purpose, we define for all  $\bsk \in \NN_0^s$, the vectors
$e_{\bsk} = e_{\bsk,\alpha,\bsgamma} := \sqrt{r_{s,\alpha,\bsgamma}(\bsk)} \, H_{\bsk}$. Note that 
$\|e_{\bsk}\|_{r_{s,\alpha,\bsgamma}}=1$.

\begin{lemma}\label{le:eigpair}
The sequence of eigenpairs of the operator $W_s$ is  $(r_{s,\alpha,\bsgamma}(\bsk), e_{\bsk})_{\bsk \in \NN_0^s}$.
\end{lemma}

\begin{proof}
	We find that for any $f,g \in \cH_{r_{s,\alpha,\bsgamma}}$ we have
	\begin{equation*}
		\langle \APP_s(f), \APP_s(g) \rangle_{L_2(\RR^s,\varphi_s)}
		=
		\langle f, \APP_s^\ast \APP_s(g) \rangle_{s,\alpha,\bsgamma}
		=
		\langle f, W_s(g) \rangle_{s,\alpha,\bsgamma}
	\end{equation*}
	and hence, due to the orthonormality of the Hermite basis functions,
	\begin{align*}
		\langle e_{\bsk}, W_s(e_{\bsh}) \rangle_{s,\alpha,\bsgamma}
		&=
		\langle e_{\bsk}, e_{\bsh} \rangle_{L_2(\RR^s,\varphi_s)}
		=
		\sqrt{r_{s,\alpha,\bsgamma}(\bsk)} \sqrt{r_{s,\alpha,\bsgamma}(\bsh)} \, \langle H_{\bsk}, H_{\bsh} \rangle_{L_2(\RR^s,\varphi_s)}
		\\
		&=
		\sqrt{r_{s,\alpha,\bsgamma}(\bsk) \, r_{s,\alpha,\bsgamma}(\bsh)} \, \delta_{\bsk,\bsh}
		.
	\end{align*}
	For $\bsk=\bsh$ this gives $\langle e_{\bsk}, W_s(e_{\bsk}) \rangle_{s,\alpha,\bsgamma} = r_{s,\alpha,\bsgamma}(\bsk)$ which in turn implies that	\begin{equation*}
		W_s(e_{\bsh})
		=
		\sum_{\bsk \in \NN_0^s} \langle W_s(e_{\bsh}), e_{\bsk} \rangle_{s,\alpha,\bsgamma} \, e_{\bsk}
		=
		r_{s,\alpha,\bsgamma}(\bsh) \, e_{\bsh}
	\end{equation*}
	and thus proves the lemma.
\end{proof}

In order to exploit the relationship between the eigenvalues of $W_s$ and the information complexity, we define the set
\begin{equation} \label{eq:set_A}
	\cA(\varepsilon, s)
	:=
	\{ \bsk \in \NN_0^s \ : \ r_{s,\alpha,\bsgamma}(\bsk) > \varepsilon^2 \}
	.
\end{equation}
Then we obtain from \eqref{infcomplam} and Lemma~\ref{le:eigpair} that 
\begin{equation}\label{infcomp1}
	n(\varepsilon,\APP_s;\Lambda^{{\rm all}})
	=
	|\cA(\varepsilon, s)|
	,
\end{equation}
which motivates us to examine the set $\cA(\varepsilon, s)$ more closely in the following lemma, which is inspired by \cite[Lemma~1]{KSW06}. 

\begin{lemma}\label{le:est-a-zeta}
Let $s \in \NN$, $\varepsilon \in (0,1)$ and let the weights satisfy $1 \ge \gamma_1 \ge \gamma_2 \ge \ldots > 0$. If $q\in \RR$ with $q > 1/\alpha$, then
\begin{equation}\label{le:it3}
|\cA(\varepsilon, s)| \le \varepsilon^{-2q} \prod_{j=1}^s \left(1 + \alpha^{\alpha q} \,\zeta(\alpha q) \, \gamma_j^q \right),
\end{equation}
where $\zeta$ denotes the Riemann zeta function.
\end{lemma}

\begin{proof}
We prove the result by induction on $s$. 

Let $s=1$. Using the upper estimate in Lemma~\ref{le:bdrk}, for any $k \in \cA(\varepsilon, 1)\setminus \{0\}$ we have
\begin{equation*}
\varepsilon^2 < r_{\alpha, \bsgamma}(k) \le \gamma_1 \left(\frac{\alpha}{k}\right)^\alpha,
\end{equation*}
which in turn gives  that $1 \le k \le \alpha \left( \frac{\gamma_1}{\varepsilon^2} \right)^{1/\alpha}$. Hence we find that
\begin{equation}\label{indanf}
|\cA(\varepsilon, 1)|\le 1 + \left|\left\{1,\ldots, \left\lfloor\alpha \left( \frac{\gamma_1}{\varepsilon^2} \right)^{1/\alpha}\right\rfloor \right\}\right| = 1 + \left\lfloor \alpha \left( \frac{\gamma_1}{\varepsilon^2} \right)^{1/\alpha}\right\rfloor.
\end{equation}
Now assume first that $\alpha \left( \gamma_1 / \varepsilon^2 \right)^{1/\alpha} \ge 1$. Then we obtain from \eqref{indanf} that 
\begin{align*}
|\cA(\varepsilon, 1)| &\le 1 + \alpha \left( \frac{\gamma_1}{\varepsilon^2} \right)^{1/\alpha} \le 1 + \left(\alpha \left( \frac{\gamma_1}{\varepsilon^2} \right)^{1/\alpha}\right)^{\alpha q} = 1 + \alpha^{\alpha q} \left( \frac{\gamma_1}{\varepsilon^2} \right)^{q}\\
&=  \left( \varepsilon^{2q} + \alpha^{\alpha q} \gamma_1^q \right) \varepsilon^{-2q} \le \left(1 + \alpha^{\alpha q} \,\zeta(\alpha q) \gamma_1^q \right) \varepsilon^{-2q}
\end{align*}
for all $q > 1/\alpha$, where we used that $\zeta(x) > 1$ for all $x > 1$. If, on the other hand, we assume that $\alpha \left( \gamma_1 / \varepsilon^2 \right)^{1/\alpha} < 1$, then we trivially have
\begin{equation*}
|\cA(\varepsilon, 1)| \le 1 + \left\lfloor\alpha \left( \frac{\gamma_1}{\varepsilon^2} \right)^{1/\alpha}\right\rfloor = 1 \le 		\left(1 + \alpha^{\alpha q} \,\zeta(\alpha q) \gamma_1^q \right) \varepsilon^{-2q}.
\end{equation*}
Thus the result is shown for $s=1$.		

Now assume that the statement holds true for $s \in \NN$ and arbitrary $\varepsilon$. First we show the recurrence
\begin{equation}\label{recA}
|\cA(\varepsilon, s+1)| \le |\cA(\varepsilon, s)| + \sum_{k_{s+1}=1}^\infty \left|\cA\left(\frac{\varepsilon}{\sqrt{\gamma_{s+1}}} \left(\frac{k_{s+1}}{\alpha}\right)^{\alpha/2}, s \right)\right|.
\end{equation}
To this end, assume that $\bsk = (k_1,\ldots, k_{s+1}) \in \cA(\varepsilon, s+1)$. Then we have
\begin{equation*}
r_{s+1,\alpha,\bsgamma}(\bsk) = r_{s,\alpha,\bsgamma}(\bsk_{[s]}) \, r_{\alpha,\gamma_{s+1}}(k_{s+1}) > \varepsilon^2 .
\end{equation*}
If $k_{s+1}=0$, then $r_{\alpha,\gamma_{s+1}}(k_{s+1})=1$ and so $r_{s,\alpha,\bsgamma}(\bsk_{[s]}) > \varepsilon^2$, that is, $\bsk_{[s]} \in \cA(\varepsilon, s)$. If on the other hand $k_{s+1} > 0$, we see that 
\begin{equation*}
r_{s,\alpha,\bsgamma}(\bsk_{[s]}) > r^{-1}_{\alpha,\gamma_{s+1}}(k_{s+1}) \, \varepsilon^2 \ge \frac{\varepsilon^2}{\gamma_{s+1}} \left(\frac{k_{s+1}}{\alpha}\right)^{\alpha}
\end{equation*}
by Lemma~\ref{le:bdrk}.
Combining both observations yields 

\begin{eqnarray*}
\lefteqn{\cA(\varepsilon, s+1)}\\
&\subseteq & \big\{(\bsk,0) \ : \ \bsk \in \cA(\varepsilon, s)\big\} \,\dot\cup\,\bigcup_{k_{s+1}=1}^{\infty} \left\{ (\bsk,k_{s+1}) \ : \  \bsk \in \cA\left(\frac{\varepsilon}{\sqrt{\gamma_{s+1}}} \left(\frac{k_{s+1}}{\alpha}\right)^{\alpha/2}, s\right) \! \right\}  ,
\end{eqnarray*}
 
where $\dot\cup$ indicates a disjoint union. From here \eqref{recA} follows immediately.
	
Now, using the recurrence formula \eqref{recA} and the induction hypothesis we obtain
\begin{eqnarray*}
\lefteqn{	|\cA(\varepsilon, s+1)|}\\
&\le& |\mathcal{A}(\varepsilon,s)| +\sum_{k=1}^\infty\left|\cA\left(\frac{\varepsilon}{\sqrt{\gamma_{s+1}}} \left(\frac{k}{\alpha}\right)^{\alpha/2}, s \right)\right|\\
&\le& \varepsilon^{-2q} \prod_{j=1}^s \left(1 + \alpha^{\alpha q} \,\zeta(\alpha q) \, \gamma_j^q \right)+\prod_{j=1}^s \left(1 + \alpha^{\alpha q} \zeta(\alpha q) \,\gamma_j^q \right) \sum_{k=1}^\infty \left[\frac{\varepsilon}{\sqrt{\gamma_{s+1}}} \left(\frac{k}{\alpha}\right)^{\alpha/2}\right]^{-2q}\\
&=& \varepsilon^{-2q} \prod_{j=1}^s \left(1 + \alpha^{\alpha q} \,\zeta(\alpha q) \, \gamma_j^q \right) \left( 1 + \gamma_{s+1}^q \alpha^{\alpha q} \sum_{k=1}^\infty \frac{1}{k^{\alpha q}} \right)\\
&=& \varepsilon^{-2q} \prod_{j=1}^{s+1} \left(1 + \alpha^{\alpha q} \,\zeta(\alpha q) \,\gamma_j^q \right).
\end{eqnarray*}
This finishes the proof.
\end{proof}

\begin{proof}[Proof of Theorem~\ref{thm:main}]
We prove the necessary and sufficient conditions for each of the listed notions of tractability.
\begin{enumerate} 
\item In order to give a necessary and sufficient condition for SPT for $\Lambda^{\rm all}$  
we use a criterion from \cite[Section~5.1]{NW08}. From \cite[Theorem 5.2]{NW08} we find that the problem $\APP$ is SPT for $\Lambda^{\rm all}$ if and only if there exists a $\tau>0$ such that 
\begin{equation}\label{critNW08} 
\sup_{s \in \NN} \left(\sum_{\bsk \in \NN_0^s} (r_{s,\alpha,\bsgamma}(\bsk))^\tau \right)^{1/\tau} < \infty
\end{equation}
and then $$\tau^\ast(\Lambda^{\mathrm{all}})=\inf\{2 \tau \ : \ \tau \text{ satisfies \eqref{critNW08}}\}.$$

Assume that $s_{\bsgamma} < \infty$. Then take $\tau$ such that $\tau > \max(s_{\bsgamma},1/\alpha)$ and thus $\sum_{j=1}^\infty \gamma_j^\tau$ is finite. Note that $\alpha \tau >1$ and hence $\zeta(\alpha \tau)< \infty$. For the sum in \eqref{critNW08} we then obtain, making use of the upper estimate in Lemma~\ref{le:bdrk},
\begin{align*}
	 \sum_{\bsk \in \NN_0^s} (r_{s,\alpha,\bsgamma}(\bsk))^{\tau}
	 &= \prod_{j=1}^{s}\left(\sum_{k=0}^\infty (r_{\alpha,\gamma_j}(k))^\tau \right)
	 \le \prod_{j=1}^{s}\left( 1 + \sum_{k=1}^\infty \gamma_j^\tau \left( \frac{\alpha}{k} \right)^{\alpha \tau} \right) \\
	 &\le \prod_{j=1}^{s}\left(1 + \gamma_j^\tau \alpha^{\alpha \tau} \zeta(\alpha \tau) \right)
	 \le \exp\left( \alpha^{\alpha \tau} \zeta(\alpha \tau) \sum_{j=1}^\infty \gamma_j^\tau\right)
	 < \infty
	 .
\end{align*}
 This implies that we have SPT and that 
\begin{equation}\label{tstup}
\tau^\ast(\Lambda^{\mathrm{all}}) \le 2 \max\left(s_{\bsgamma},\frac{1}{\alpha}\right).
\end{equation}

On the other hand, assume we have SPT. Then there exists a finite $\tau$ such that \eqref{critNW08} holds true. 
Using the lower bound in Lemma~\ref{le:bdrk}, we have that
\begin{eqnarray*}
\sum_{\bsk \in \NN_0^s} (r_{s,\alpha,\bsgamma}(\bsk))^\tau = \prod_{j=1}^{s}\left(\sum_{k=0}^\infty (r_{\alpha,\gamma_j}(k))^\tau \right) \ge  \prod_{j=1}^{s}\left(1 + \gamma_j^\tau  \sum_{k=1}^\infty \frac{1}{k^{\alpha \tau}}\right).
\end{eqnarray*}
Since \eqref{critNW08} holds true we obviously have   $\tau > 1/\alpha$. Then
\begin{eqnarray*}
\sum_{\bsk \in \NN_0^s} (r_{s,\alpha,\bsgamma}(\bsk))^\tau \ge \prod_{j=1}^{s} (1 + \gamma_j^\tau \zeta(\alpha \tau)) \ge  \zeta(\alpha \tau) \sum_{j=1}^s \gamma_j^\tau.
\end{eqnarray*}
Again, since \eqref{critNW08} holds true, we also have   that $\sum_{j=1}^\infty \gamma_j^\tau<\infty$ and hence $s_{\bsgamma}< \tau < \infty$. 
Combining both results yields that $\tau > \max(s_{\bsgamma},1/\alpha)$ and hence also 
\begin{equation}\label{tstdn}
\tau^\ast(\Lambda^{\mathrm{all}}) \ge 2 \max\left(s_{\bsgamma},\frac{1}{\alpha}\right).
\end{equation}
Equations \eqref{tstup} and \eqref{tstdn} then imply that $$\tau^\ast(\Lambda^{\mathrm{all}}) = 2 \max\left(s_{\bsgamma},\frac{1}{\alpha}\right).$$

\item In order to prove the equivalence of SPT and PT it suffices to prove that PT implies SPT. So let us assume that $\APP$ is polynomially tractable, i.e., there exist reals $C,p>0$ and $q \ge 0$ such that $$n(\varepsilon,\APP_s;\Lambda^{{\rm all}}) \le C s^q \varepsilon^{-p} \quad \mbox{for all $\varepsilon \in (0,1)$ and for all $s \in \NN$}.$$ Without loss of generality we may assume that $q$ is an integer. Take $s \in \NN$ such that $s \ge q+1$ and 
set \[
B_s:=\{\bsh \in \{0,1\}^s\colon \text{precisely $q+1$ components of $\bsh$ are equal to $1$}\}\,.
\]
Now choose $\varepsilon_*=\frac{1}{2} \gamma_s^{(q+1)/2}$. Choose $\bsh \in B_s$  and let $\uu \subseteq [s]$ be the set of indices of $\bsh$ which are equal to $1$. Then we have 
$$r_{s,\alpha,\bsgamma}(\bsh) \ge  \prod_{j \in \uu} \gamma_j \ge \gamma_s^{q+1} > \varepsilon_*^2.$$
Hence $B_s\subseteq\cA(\varepsilon_*,s)$ and this implies $$|\cA(\varepsilon_*,s)| \ge|B_s|= {s \choose q+1} \ge \frac{(s-q)^{q+1}}{(q+1)!} \ge \frac{s^{q+1}}{(q+1)! (q+1)^{q+1}}=: s^{q+1} c_q\,,$$
where we used $s\ge q+1$ for the third inequality.
 
This now yields
\begin{eqnarray*}
s^{q+1} c_q \le |\cA(\varepsilon_*,s)|  = n(\varepsilon^*,\APP_s;\Lambda^{{\rm all}}) \le C s^q \varepsilon_*^{-p} = 2^p C s^q \gamma_s^{-(q+1)p/2}.  
\end{eqnarray*}
This implies that $$\gamma_s^{(q+1)p/2} \ll_{p,q} \frac{1}{s},$$ where $\ll_{p,q}$ means that there is an implied factor which only depends on $p$ and $q$, and hence $$\gamma_s \ll_{p,q} \frac{1}{s^{2/((q+1)p)}}.$$ This estimate holds for all $s \ge q+1$. Hence the sum exponent $s_{\bsgamma}$ of the sequence $\bsgamma=(\gamma_j)_{j \ge 1}$ is finite, $s_{\bsgamma} < \infty$, and this implies by the first statement that we have SPT.

\item We use the following criterion for QPT taken from \cite[Sec.~23.1.1]{NW12}, which states that QPT holds if and only if there exists $\tau>0$ such that 
\begin{equation} \label{condQPT}
	C
	:=
	\sup_{s \in \NN} \frac{1}{s^2} \left(\sum_{j=1}^{\infty} \lambda_{s,j}^{\tau(1+\log s)} \right)^{1/\tau} 
	< 
	\infty,
\end{equation}
where $\lambda_{s,j}$ is the $j$-th eigenvalue of the operator $W_s$ from \eqref{def:Ws} in non-increasing order. 

Assume that $\bsgamma_I<1$. For the considered Hermite space $\cH_{r_{s,\alpha,\bsgamma}}$ we have
\begin{align}\label{condQPT2}
\lefteqn{	\sum_{j=1}^{\infty} \lambda_{s,j}^{\tau(1+\log s)}}\nonumber\\
	&=
	\sum_{\bsk \in \NN_0^s} (r_{s,\alpha,\bsgamma}(\bsk))^{\tau (1 + \log s)}\nonumber\\
	&=
	\prod_{j=1}^s \left( 1 + \sum_{k=1}^\infty (r_{\alpha,\gamma_j}(k))^{\tau (1 + \log s)} \right)\nonumber
	\\
	&=
	\prod_{j=1}^s \left( 1 + \gamma_j^{\tau (1 + \log s)} \left( \sum_{k=1}^{\alpha -1} \left(\frac{1}{k!}\right)^{\tau (1 + \log s)} 
	+ \sum_{k=\alpha}^\infty \left(\frac{(k-\alpha)!}{k!}\right)^{\tau (1 + \log s)} \right) \right)
	\\
	&\le
	\prod_{j=1}^s \left( 1 + \gamma_j^{\tau (1 + \log s)} \left(\sum_{k=1}^{\alpha -1} \left(\frac{1}{k!}\right)^{\tau (1 + \log s)}  + \sum_{k=\alpha}^\infty \left(\frac1{(k-\alpha+1)^\alpha}\right)^{\tau (1 + \log s)} \right) \right)\nonumber
	\\
	&=
	\prod_{j=1}^s \left( 1 + \gamma_j^{\tau (1 + \log s)} \left(\sum_{k=1}^{\alpha -1} \left(\frac{1}{k!}\right)^{\tau (1 + \log s)}  + \sum_{k=1}^\infty \left(\frac1{k^\alpha}\right)^{\tau (1 + \log s)} \right) \right)\nonumber
	\\
	&=
	\prod_{j=1}^s \left( 1 + \gamma_j^{\tau (1 + \log s)} \left(\sum_{k=1}^{\alpha -1} \left(\frac{1}{k!}\right)^{\tau (1 + \log s)}  + \zeta(\alpha \tau (1 + \log s))  \right) \right)
	.\nonumber
\end{align}
Put $\zeta_s:=\zeta(\alpha \tau (1 + \log s))$. From now on we assume, without loss of generality, that $\tau>1/\alpha$, so that $\zeta_s < \infty$. 

Next we have
\begin{eqnarray*}
\sum_{k=1}^{\infty} \left(\frac{1}{k!}\right)^{\tau (1 + \log s)} & = & 1+\frac{1}{2^{\tau(1+\log s)}}+\frac{1}{6^{\tau(1+\log s)}} +\sum_{k=4}^{\infty} \left(\frac{1}{k!}\right)^{\tau (1 + \log s)}\\
 & \le & 1+\frac{1}{2^{\tau(1+\log s)}}+\frac{1}{6^{\tau(1+\log s)}} +\sum_{k=4}^{\infty} \left(\frac{1}{2^{\tau(1 + \log s)}}\right)^k\\
 & = & 1+\frac{1}{2^{\tau(1+\log s)}}+\frac{1}{6^{\tau(1+\log s)}} + \frac{1}{2^{4 \tau (1 + \log s)}} \frac{1}{1-1/2^{\tau(1+\log s)}}\\
 & = & 1+\frac{1}{2^{\tau(1+\log s)}}+\frac{1}{6^{\tau(1+\log s)}} +  \frac{1}{2^{4 \tau (1 + \log s)}-2^{3 \tau (1+\log s)}}\\
 & \le & 1+\frac{1}{2^{\tau(1+\log s)}}+\frac{1}{6^{\tau(1+\log s)}} + \frac{1}{2^{3 \tau \log s}}\ \frac{1}{2^{4 \tau}-2^{3\tau}}\\
 & \le & 1+ \frac{1}{2^{\tau \log s}} \left(2+ \frac{1}{2^{4 \tau}-2^{3\tau}}\right)\\
 & = & 1+ \frac{c_{\tau}}{s^{\tau \log 2}}, 
\end{eqnarray*}
where we put $c_{\tau}:=2+ 1/(2^{4\tau}-2^{3 \tau})$. This in turn gives that
\begin{align*}
\frac{1}{s^2} \left(\sum_{j=1}^{\infty} \lambda_{s,j}^{\tau(1+\log s)} \right)^{1/\tau}
&\le
\frac{1}{s^2} \left( \prod_{j=1}^s \left( 1 + \gamma_j^{\tau (1 + \log s)} \left(  \frac{c_{\tau}}{s^{\tau \log 2}}+ \zeta_s \right) \right) \right)^{1/\tau}
\\
&=
\exp\left( \frac1{\tau} \sum_{j=1}^s \log\left( 1 + \gamma_j^{\tau (1 + \log s)} \left( \zeta_s+  \frac{c_{\tau}}{s^{\tau \log 2}}\right) \right) - 2 \log s \right)
\\
&\le
\exp\left( \frac1{\tau} \, \left( \zeta_s+  \frac{c_{\tau}}{s^{\tau \log 2}}\right) \sum_{j=1}^s \gamma_j^{\tau (1 + \log s)} - 2 \log s \right)
,
\end{align*}
where we used that $\log(1+x) \le x$ for all $x \ge 0$. Now we use the well-known fact that $\zeta(x) \le 1 + \frac1{x-1}$ for all $x>1$ and thus
\begin{equation*}
\zeta_s \le 1 + \frac1{(\alpha \tau - 1) + \alpha \tau \log s}.
\end{equation*}
Then we obtain
\begin{align*}
\lefteqn{\frac{1}{s^2} \left(\sum_{j=1}^{\infty} \lambda_{s,j}^{\tau(1+\log s)} \right)^{1/\tau}}\\
&\le \exp\left( \frac1{\tau} \, \left(1 +  \frac1{(\alpha \tau - 1) + \alpha \tau \log s} + \frac{c_{\tau}}{s^{\tau \log 2}}  \right) \sum_{j=1}^s \gamma_j^{\tau (1 + \log s)} - 2 \log s \right).
\end{align*}

Now we distinguish  two cases:
	\begin{itemize} 
	\item Case $\bsgamma_I=0$: Then $\lim_{j \rightarrow \infty}\gamma_j=0$ and hence, for every $\varepsilon >0$ there is a positive integer $J=J(\varepsilon)$ such that $\gamma_j \le \varepsilon$ for all $j \ge J$. Then 
	\begin{eqnarray*}
	 \sum_{j=1}^s \gamma_j^{\tau (1 + \ln s)} & \le & \sum_{j=1}^{J-1} 1+ \sum_{j =J}^s \varepsilon^{\tau \ln s} \le J-1+ s^{1-\tau \ln \varepsilon^{-1}}. 
	\end{eqnarray*}
	Choosing $\varepsilon = \exp(-1/\tau)$ we obtain $$\sum_{j=1}^s \gamma_j^{\tau (1 + \ln s)}  \le J.$$ Note that the last $J$ depends on $\tau$, but it is finite for every fixed $\tau$.  Thus, if $\tau > 1/\alpha$ and $\bsgamma_I =0$ we have 
	\begin{eqnarray*}
		\lefteqn{\frac{1}{s^2} \left(\sum_{j=1}^{\infty} \lambda_{s,j}^{\tau(1+\ln s)} \right)^{1/\tau}}\\
		&\le&
		\exp\left( \frac1{\tau} \, \left(1 +  \frac1{(\alpha \tau - 1) + \alpha \tau \log s} + \frac{c_{\tau}}{s^{\tau \log 2}}  \right) J - 2 \log s \right)
		\\
		&\rightarrow& 0 \quad \text{if $s\to \infty$.}
	\end{eqnarray*}
	
	By the characterization in \eqref{condQPT}, this implies QPT.
      \item Case $\bsgamma_I \in (0,1)$:  then for every real $\gamma_{\ast} \in (\bsgamma_I,1)$ there exists a $j_0=j_0(\gamma_{\ast}) \in \NN$ such that $$\gamma_j \le \gamma_{\ast} < 1 \quad \mbox{for all}\  j > j_0.$$ Then we obtain for every $s \in \NN$ that 
	\begin{align*}
		\sum_{j=1}^s \gamma_j^{\tau(1+\ln s)} 
		&\le
		j_0 +  \gamma_\ast^{\tau(1+ \ln s)} \max(s-j_0,0) \\
		&=
		j_0 +\frac{\gamma_\ast^{\tau} \max(s-j_0,0)}{s^{\tau \ln \gamma_\ast^{-1}}}
		\le 
		j_0+1
		,
	\end{align*}
	as long as $\tau \ge (\ln \gamma_\ast^{-1})^{-1}$. Thus, if $\tau > 1/\alpha$ and $\tau \ge (\ln \gamma_\ast^{-1})^{-1}$ we have 
	\begin{eqnarray*}
		\lefteqn{\frac{1}{s^2} \left(\sum_{j=1}^{\infty} \lambda_{s,j}^{\tau(1+\ln s)} \right)^{1/\tau}}\\
		&\le&
		\exp\left( \frac1{\tau} \, \left(1 +  \frac1{(\alpha \tau - 1) + \alpha \tau \log s} + \frac{c_{\tau}}{s^{\tau \log 2}}  \right) (j_0+1) - 2 \log s \right)
		\\
		&\rightarrow& 0 \quad \text{if $s \to \infty$.}
	\end{eqnarray*}
	Again, by the characterization in \eqref{condQPT}, this implies QPT. 
      \end{itemize}

Of course, QPT implies UWT and this in turn implies WT. 

So it suffices to show that WT implies $\bsgamma_I<1$. Assume on the contrary that $\bsgamma_I=1$, i.e., $\gamma_j=1$ for all $j \in \NN$. Then for all $\bsk \in \{0,1\}^s$ we have $r_{s,\alpha,\bsgamma}(\bsk)=1$. This means that $\{0,1\}^s \subseteq \mathcal{A}(\varepsilon,s)$, where $\mathcal{A}(\varepsilon,s)$ is defined in \eqref{eq:set_A}, and hence, according to \eqref{infcomp1}, $n(\varepsilon,\APP_s;\Lambda^{{\rm all}}) \ge 2^s$. This means that the approximation problem suffers from the curse  of dimensionality and, in particular, we cannot have WT. This proves the first assertion of item 3. 

It remains to show the result about the exponent of QPT. Again from \cite[Sec.~23.1.1]{NW12} we know that the exponent of QPT is $$t^{\ast}(\Lambda^{{\rm all}})=2 \inf\{\tau \ : \ \tau \text{ such that \eqref{condQPT} holds}\}.$$ 
From the first part of the proof of item 3.~it follows that $\tau$ satisfies \eqref{condQPT} as long as 
\[
\tau>
\begin{cases}
\max(\frac{1}{\alpha},(\ln \bsgamma_I^{-1})^{-1})& \text{ if } \bsgamma_I\ne 0,\\
\frac{1}{\alpha}& \text{ if } \bsgamma_I= 0.
\end{cases}
\]

Therefore, 
	\begin{equation*}
		t^{\ast}(\Lambda^{{\rm all}})
		\le 
		2 \max\left(\frac{1}{\alpha} ,\frac{1}{\ln \bsgamma_I^{-1}}\right)
		.
	\end{equation*}

Assume now that we have QPT. Then \eqref{condQPT} holds true for some $\tau>0$. Considering the special instance $s=1$ this means $$C\ge \left(\sum_{j=1}^{\infty} \lambda_{1,j}^{\tau} \right)^{1/\tau}.$$ According to \eqref{condQPT2} we then have 
\begin{align*}
C \ge \left(1+\gamma_1^{\tau} \left(\sum_{k=\alpha+1}^{\infty} \left(\frac{(k-\alpha)!}{k!}\right)^{\tau}\right)\right)^{1/\tau} \ge \gamma_1 \left(\sum_{k=\alpha+1}^{\infty} \frac{1}{k^{\alpha \tau}}\right)^{1/\tau}
\end{align*}
and hence we must have $\tau>1/\alpha$. This already implies the result $t^{\ast}(\Lambda^{{\rm all}})
		= 2/\alpha$, whenever $\bsgamma_I=0$.

It remains to study the case $\bsgamma_I>0$. Now, again according to \eqref{condQPT} and \eqref{condQPT2}, there exists a $\tau>1/\alpha$ such that for all $s \in \NN$ we have 
\begin{align*}
C  \ge \frac{1}{s^2} \left( \prod_{j=1}^s \left(1+\gamma_j^{(\tau(1+\log s)} \right)\right)^{1/\tau} = \exp\left(\frac{1}{\tau} \sum_{j=1}^s \log\left(1+\gamma_j^{\tau(1+\log s)}\right) -2 \log s \right).
\end{align*}
Taking the logarithm leads to 
\begin{eqnarray*}
\log C \ge \frac{1}{\tau} \sum_{j=1}^s \log\left(1+\gamma_j^{\tau(1+\log s)}\right) - 2 \log s  \ge  \frac{s}{\tau} \log\left(1+\bsgamma_I^{\tau(1+\log s)}\right) - 2 \log s
\end{eqnarray*}
for all $s \in \NN$. Since $\bsgamma_I \in (0,1)$ and since $\log(1+x)\ge x \log 2$ for all $x \in [0,1]$ it follows that for all $s \in \NN$ we have $$\log C \ge \frac{s \log 2}{\tau}  \bsgamma_I^{\tau(1+\log s)} -2 \log s = \frac{\bsgamma_I^{\tau}  s \log 2}{\tau \, s^{\tau \log \bsgamma_I^{-1}}}  -2 \log s.$$ This implies that $\tau \ge (\log \bsgamma_I^{-1})^{-1}$. Therefore, we also have that $$t^{\ast}(\Lambda^{{\rm all}}) \ge 2 \max\left(\frac{1}{\alpha} ,\frac{1}{\log \bsgamma_I^{-1}}\right).$$ 
Hence  item 3.~is proven . 
\item We know from \eqref{infcomp1} that $n(\varepsilon,\APP_s) = |\cA(\varepsilon, s)|$. Fixing some $q>1/\alpha$, we get from Lemma \ref{le:est-a-zeta}  
\begin{eqnarray*}
n(\varepsilon,\APP_s;\Lambda^{{\rm all}})  & = & |\cA(\varepsilon, s)|\\
& \le & \varepsilon^{-2q} \prod_{j=1}^s \left(1 + \alpha^{\alpha q} \,\zeta(\alpha q) \, \gamma_j^q \right)\\
& \le & \varepsilon^{-2q} \left(1 + \alpha^{\alpha q} \,\zeta(\alpha q)\right)^s.
\end{eqnarray*}
Write $c:=1 + \alpha^{\alpha q} \,\zeta(\alpha q)$ which is larger than $1$. Then we have $$\log n(\varepsilon,\APP_s;\Lambda^{{\rm all}})  \le 2q \log \varepsilon^{-1} +s \log c.$$ Hence, for $\sigma>1$ 
$$\lim_{s+\varepsilon^{-1}\rightarrow \infty} \frac{\log n(\varepsilon,\APP_s;\Lambda^{{\rm all}})}{s^{\sigma}+\varepsilon^{-\tau}} \le \lim_{s+\varepsilon^{-1}\rightarrow \infty} \frac{2q \log \varepsilon^{-1} +s \log c}{s^{\sigma}+\varepsilon^{-\tau}}=0.$$ This implies $(\sigma,\tau)$-WT.
\end{enumerate}
\end{proof}

\begin{proof}[Proof of Corollary~\ref{cor:app:R:std}]
From Proposition~\ref{pr:r_rho} and \ref{pr:r_psi} in conjunction with Proposition~\ref{pr:embb:err} we find that for $R\in \{\rho_{s,\alpha,\bsgamma},\psi_{s,\alpha,\bsgamma}\}$ and for all $n \in \NN$ we have $$e(n,\APP_R;\Lambda^{{\rm all}}) \le e(n,\APP_{r_{s,\alpha,\bsgamma}};\Lambda^{{\rm all}})$$ and hence $$n(\varepsilon,\APP_R;\Lambda^{{\rm all}}) \le n(\varepsilon,\APP_{r_{s,\alpha,\bsgamma}};\Lambda^{{\rm all}})\quad \mbox{for all $\varepsilon \in (0,1)$.}$$ Thus the sufficient conditions follow directly from Theorem~\ref{thm:main}.

Let $R\in \{\rho_{s,\alpha,\bsgamma},\psi_{s,\alpha,\bsgamma}\}$ and assume that we have (S)PT for $L_2$-approximation in $\cH_R$ for the class $\Lambda^{{\rm all}}$. Then it follows from Proposition~\ref{pr:r_rho} and \ref{pr:r_psi} in conjunction with Proposition~\ref{pr:embb:err} that we have (S)PT for $L_2$-approximation in $\cH_{r_{s,\alpha,\bsgamma/(t \alpha^{\alpha})}}$ for the class $\Lambda^{{\rm all}}$, where $t=1$ if $R=\rho_{s,\alpha,\bsgamma}$ and $t=2$ if  $R=\psi_{s,\alpha,\bsgamma}$. From Theorem~\ref{thm:main} we now obtain $s_{\bsgamma/(t \alpha^{\alpha})}< \infty$. Since $s_{\bsgamma}=s_{\bsgamma/(t \alpha^{\alpha})}$ this implies $s_{\bsgamma}< \infty$.
\end{proof}

\subsection{Tractability for the class $\Lambda^{\text{std}}$}

The next theorem states sufficient conditions for tractability of $L_2$-approximation for the class $\Lambda^{{\rm std}}$.

\begin{theorem} \label{thm:tract-standard}
Let $\alpha > 1$ and $\bsgamma$ be a sequence of weights. Consider multivariate approximation $\APP = (\APP_s)_{s \ge 1}$ for the weighted Hermite spaces $\cH_R$, $R\in \{r_{s,\alpha,\bsgamma},\rho_{s,\alpha,\bsgamma},\psi_{s,\alpha,\bsgamma}\}$ for $s \in \NN$ and for the information class $\Lambda^{{\rm std}}$. Then we have the following sufficient conditions:
\begin{enumerate}
\item SPT holds if $$\sum_{j =1}^{\infty} \gamma_j < \infty.$$
In this case the exponent of SPT satisfies 
\begin{equation}\label{cond_exp_spt_std}		
		\tau^{\ast}(\Lambda^{\mathrm{std}}) = 2 \max\left(\frac{1}{\alpha},s_{\bsgamma}\right).
		\end{equation}
\item PT holds if 
\begin{equation}\label{cond_pt_std}	
	 \limsup_{s \to \infty} \frac1{\ln s} \sum_{j=1}^s \gamma_j < \infty.
\end{equation}
\item WT holds if
		\begin{equation}\label{cond_wt_std}
			\lim_{s \to \infty} \frac1{s} \sum_{j=1}^s \gamma_j = 0.
		\end{equation}
\item For $\sigma\in (0,1]$ $(\sigma,\tau)$-WT holds if
		\begin{equation}\label{cond_tswt_std}
			\lim_{s \to \infty} \frac1{s^\sigma} \sum_{j=1}^s \gamma_j = 0.
		\end{equation}
\item UWT holds if 
		\begin{equation}\label{cond_uwt_std}
			\lim_{s \to \infty} \frac1{s^\sigma} \sum_{j=1}^s \gamma_j = 0
			\quad \text{for all } \sigma \in (0,1].
		\end{equation}
\end{enumerate}
\end{theorem}

It suffices to prove the result for $R=r_{s,\alpha,\bsgamma}$. Our analysis will be based on relations between the minimal errors of $\Lambda^{\text{std}}$ and $\Lambda^{\text{all}}$, in particular on \cite[Theorem~1]{DKU} and on \cite[Theorem~1]{WW01} (see also \cite[Theorem~26.10]{NW12}).  These results provide that the trace of the operator $W_s$ from \eqref{def:Ws} is finite. 
Recall that the trace of $W_s$ is given by the sum of its eigenvalues, that is, 
\begin{align*}
	\tra(W_s)
	&=
	\sum_{j=1}^{\infty} \lambda_{s,j}
	=
	\sum_{\bsk \in \NN_0^s} r_{s,\alpha,\bsgamma}(\bsk)
	=
	\prod_{j=1}^s \left( 1 + \sum_{k=1}^\infty r_{\alpha,\gamma_j}(k) \right)
	\\
	&=
	\prod_{j=1}^s \left( 1 + \gamma_j \left( \sum_{k=1}^{\alpha-1 }\frac1{k!} + \sum_{k=\alpha}^\infty \frac{(k-\alpha)!}{k!} \right) \right)
	\\
	&\le
	\prod_{j=1}^s \left( 1 + \gamma_j \left( {\rm e} - 1 + \sum_{k=1}^\infty \frac{1}{k^\alpha}  \right) \right)	\\
	&=
	\prod_{j=1}^s \Big( 1 + \gamma_j \big( {\rm e} - 1  + \zeta(\alpha) \big) \Big)
	,	
\end{align*}
which is finite provided that $\alpha > 1$. Using Lemma~\ref{le:bdrk} we obtain in a similar way that
$$\tra(W_s) \ge \prod_{j=1}^s \big(1+\gamma_j \zeta(\alpha)\big)$$ and hence $\tra(W_s)$ is infinite if and only if $\alpha=1$. Note that in general there is no relation between the power of $\Lambda^{{\rm all}}$ and $\Lambda^{{\rm std}}$ whenever the trace of $W_s$ is infinite. For a discussion of this issue we refer to \cite[Section~26.3]{NW12}.

However, if $\alpha>1$ we obtain that there exists a positive constant 
$c(\alpha) \in [\zeta(\alpha), {\rm e} - 2 + \zeta(\alpha)]$ such that the trace of $W_s$ equals
\begin{equation}\label{fo:tra}
	\tra(W_s) 
	= 
	\prod_{j=1}^s \big( 1 + \gamma_j \, c(\alpha) \big)	
\end{equation}
and is finite for all $s \in \NN$.

\begin{proof}[Proof of Theorem~\ref{thm:tract-standard}]
According to Proposition~\ref{pr:r_rho} and \ref{pr:r_psi} in conjunction with Proposition~\ref{pr:embb:err} it suffices to proof the result for the Fourier weights $R=r_{s,\alpha,\bsgamma}$. 

Since $\alpha>1$ we know that $\tra(W_s)$ is finite for all $s \in \NN$. 
\begin{enumerate}
\item  For the proof we use \cite[Theorem~1]{DKU} from which we know that there exists a universal constant $c \in \NN$ such that for all $n \in \NN$ we have 
\begin{equation}\label{bd:DKU}
e(c\, n,\APP_s; \Lambda^{{\rm std}})^2 \le \frac{1}{n} \sum_{k=n}^{\infty} e(k,\APP_s;\Lambda^{{\rm all}})^2.
\end{equation}
Assume that $\sum_{j=1}^{\infty} \gamma_j < \infty$. Then, obviously, the sum exponent of the weight sequence $\bsgamma$ satisfies $s_{\bsgamma} \le 1$.
Assume first that $s_{\bsgamma}<1$. Then, according to Theorem~\ref{thm:main} we have SPT for $\Lambda^{{\rm all}}$ with exponent 
$$\tau^{\ast}(\Lambda^{{\rm all}}) =2 \max\left(\frac{1}{\alpha},s_{\bsgamma}\right)< 2.$$ 
Hence for every $\tau>\tau^{\ast}(\Lambda^{{\rm all}})$ there exists a $C>0$ such that $n(\varepsilon,\APP_s;\Lambda^{{\rm all}}) \le C \varepsilon^{-\tau}$ and from this we deduce $$e(k,\APP_s;\Lambda^{{\rm all}}) \le \frac{C}{k^{1/\tau}}.$$ Inserting into \eqref{bd:DKU} yields
\begin{eqnarray*}
e(c\, n,\APP_s; \Lambda^{{\rm std}})^2 & \le & \frac{C}{n} \sum_{k=n}^{\infty} \frac{1}{k^{2/\tau}}\\
& \le & \frac{C}{n} \int_{n-1}^{\infty} \frac{1}{x^{2/\tau}} \rd x \\
& = & \frac{C}{n} \frac{\tau}{2-\tau} \frac{1}{(n-1)^{2/\tau - 1}} \\
& \le & \frac{C\, \tau}{2-\tau} \frac{1}{(n-1)^{2/\tau}}.
\end{eqnarray*}
Hence there exists a number $a_{\tau}>0$ such that $$e(c\, n,\APP_s; \Lambda^{{\rm std}}) \le \frac{a_{\tau}}{n^{1/\tau}}.$$ 
This implies that $$n(\varepsilon,\APP_s;\Lambda^{{\rm std}}) \le \left\lceil c\, a_{\tau}^{\tau} \, \varepsilon^{-\tau}\right\rceil$$ and hence, since $\tau >\tau^{\ast}(\Lambda^{{\rm std}})$ was arbitrary, we have SPT with exponent $$\tau^{\ast}(\Lambda^{{\rm std}}) =  2 \max\left(\frac{1}{\alpha},s_{\bsgamma}\right).$$ (Note that trivially $\tau^{\ast}(\Lambda^{{\rm std}}) \ge \tau^{\ast}(\Lambda^{{\rm all}})=2 \max(1/\alpha,s_{\bsgamma})$.)

Now assume that $s_{\bsgamma}=1$. From \eqref{bd:DKU} and \eqref{errEW} we obtain
\begin{equation}\label{bd:DKU1}
e(c\, n,\APP_s; \Lambda^{{\rm std}})^2 \le \frac{1}{n} \sum_{k=n}^{\infty} \lambda_{s,k+1} \le \frac{1}{n} \sum_{j=1}^{\infty} \lambda_{s,j} = \frac{{\rm trace}(W_s)}{n}.
\end{equation}
Now we use \eqref{fo:tra}. For $\sum_{j=1}^{\infty}\gamma_j < \infty$ and $\alpha>1$ we have 
\begin{eqnarray*}
\tra(W_s) = \exp\left(\sum_{j=1}^s \ln(1+\gamma_j c(\alpha))\right)  \le \exp\left(c(\alpha) \sum_{j=1}^{\infty} \gamma_j\right)=:\Gamma< \infty.
\end{eqnarray*}
Hence, inserting into \eqref{bd:DKU1} gives
$$e(c\, n,\APP_s; \Lambda^{{\rm std}})^2 \le  \frac{\Gamma}{n}.$$ From this we obtain in the same way as above SPT with exponent $$\tau^{\ast}(\Lambda^{{\rm std}})=2= 2 \max\left(\frac{1}{\alpha},s_{\bsgamma}\right).$$ 

\item We will use \cite[Theorem~26.13]{NW12}.   Assume that the weights satisfy \eqref{cond_pt_std}. This implies that there exists a finite, positive $M$ such that $\frac1{\ln s} \sum_{j=1}^s \gamma_j < M$ for all $s$.
Then we have 
\begin{equation*}
\tra(W_s) \le 	\exp\left( c(\alpha) \sum_{j=1}^s \gamma_j \right) \le \exp\left(c(\alpha) M \ln s \right) = s^{c(\alpha) M}.
\end{equation*}
Furthermore, assumption \eqref{cond_pt_std} implies that
\begin{equation*}
\frac{s \,\gamma_s}{\ln s} \le \frac1{\ln s} \sum_{j=1}^s \gamma_j < M \quad \text{for all} \quad s \in \NN
\end{equation*}
and therefore $\gamma_j = \cO(j^{-1} \ln j)$ and in particular $s_{\bsgamma} = 1$. By the characterization in Theorem~\ref{thm:main} this implies that approximation is (S)PT for the class in $\Lambda^{\text{all}}$, i.e., there exist positive $C^{{\rm all}}$ and $p^{{\rm all}}$ such that
$$n(\varepsilon,\APP_s;\Lambda^{{\rm all}}) \le C^{{\rm all}} \varepsilon^{-p^{{\rm all}}} \quad \mbox{ for all $\varepsilon \in (0,1)$ and $s \in \NN$.}$$ Now \cite[Theorem 26.13]{NW12} implies the existence of a positive $C^{{\rm std}}$ such that $$n(\varepsilon,\APP_s;\Lambda^{{\rm std}}) \le C^{{\rm std}} \varepsilon^{-p^{{\rm std}}} s^{q^{{\rm std}}} \quad \mbox{ for all $\varepsilon \in (0,1)$ and $s \in \NN$,}$$ where $$p^{{\rm std}}=p^{{\rm all}}+2 \quad \mbox{ and }\quad q^{{\rm std}}= c(\alpha) M.$$ Hence we have PT also for the class $\Lambda^{{\rm std}}$.

\item[3.-5.] We prove the three statements in one combined argument.  If any of the three conditions \eqref{cond_wt_std}, \eqref{cond_tswt_std} or \eqref{cond_uwt_std} holds, then this implies that the weights $(\gamma_j)_{j\ge1}$ 
(which we assumed to be non-increasing) have to become less than $1$ eventually since otherwise, for every $\sigma \in (0,1]$,
\begin{equation*}
\lim_{s \to \infty} \frac1{s^\sigma} \sum_{j=1}^s \gamma_j = \lim_{s \to \infty} \frac{s}{s^\sigma} = \lim_{s \to \infty} s^{1-\sigma} \ge 1.
\end{equation*}
(Actually we even have $\bsgamma_I=0$.)
Therefore, we have by Theorem~\ref{thm:main} that UWT (and even QPT) holds for the class $\Lambda^{\text{all}}$. Furthermore, we observe that
\begin{align*}
\frac{\ln(\tra(W_s))}{s^\sigma} &= \frac1{s^\sigma} \ln \left( \prod_{j=1}^s (1 + \gamma_j \, c(\alpha)) \right) = \frac1{s^\sigma} \sum_{j=1}^s \ln(1 + \gamma_j \, c(\alpha)) \le \frac{c(\alpha)}{s^\sigma} \sum_{j=1}^s \gamma_j
\end{align*}   
and thus if $\frac1{s^\sigma} \sum_{j=1}^s \gamma_j$ converges to $0$ as $s$ goes to infinity, with $\sigma \in (0,1]$, then
\begin{equation} \label{eq:cond_thm_26}
\lim_{s \to \infty} \frac{\ln(\tra(W_s))}{s^\sigma} \le \lim_{s \to \infty} \frac{c(\alpha)}{s^\sigma} \sum_{j=1}^s \gamma_j = 0.
\end{equation}
Now we obtain with the same arguments as in the proof of  \cite[Theorem~26.11]{NW12} that \eqref{cond_wt_std} implies WT for the class $\Lambda^{\text{std}}$. The proof for the other two notions of WT can be obtained analogously by appropriately modifying the argument used in the proof of \cite[Theorem~26.11]{NW12}.
\end{enumerate}
The proof is complete.
\end{proof}

\begin{remark}\rm
It is obvious from \eqref{est:icomp_allstd} that the sufficient conditions for tractability for information from the class $\Lambda^{{\rm std}}$ are not weaker  than the sufficient conditions for the respective notion of tractability  for information from the class $\Lambda^{{\rm all}}$. For example SPT for the class $\Lambda^{{\rm all}}$ holds if $s_{\bsgamma}< \infty$, whereas the sufficient condition for SPT for the class $\Lambda^{{\rm std}}$ is $\sum_{j=1}^{\infty} \gamma_j<\infty$, which can be re-formulated in an equivalent way as $s_{\bsgamma} \le 1$.
\end{remark}

\begin{remark}\rm
Again from \eqref{est:icomp_allstd} it follows that every necessary condition for tractability for information from the class $\Lambda^{{\rm all}}$ is also necessary for the respective notion of tractability  for information from the class $\Lambda^{{\rm std}}$. Unfortunately these conditions do not match the sufficient conditions obtained from Theorem~\ref{thm:tract-standard}. However, it follows from the argument used in item 3 of the proof of Theorem~\ref{thm:main} that in the unweighted case, i.e., $\gamma_j=1$ for all $j \in \NN$, we have $$n(\varepsilon,\APP_s;\Lambda^{{\rm std}}) \ge n(\varepsilon,\APP_s;\Lambda^{{\rm all}}) \ge 2^s$$ and hence for the unweighted case the $L_2$-approximation problem for information from $\Lambda^{{\rm std}}$ suffers from the curse of dimensionality. 
\end{remark}

\begin{remark}\rm
While we have a very clear picture of tractability of $L_2$-approximation for the Hermite space $\cH_{r_{s,\alpha,\bsgamma}}$ for the information class $\Lambda^{{\rm all}}$ there remain several open questions concerning $\Lambda^{{\rm std}}$. In the first place, matching necessary conditions for the respective notions of tractability are still missing. Furthermore, we neither have sufficient nor necessary conditions for quasi-polynomial tractability beyond the sufficient condition for polynomial tractability which obviously also implies quasi-polynomial tractability. Finally, our results require a smoothness parameter $\alpha$ bigger than 1. Similar results for $\alpha=1$ are still missing.
\end{remark}

\section{Integration in weighted Hermite spaces}\label{sec:int}

Now we consider the integration problem. The next theorem states sufficient conditions for tractability of integration. Obviously, the information class $\Lambda^{{\rm all}}$ makes this problem trivial. For this reason we restrict to the class $\Lambda^{{\rm std}}$.

\begin{theorem}\label{thm:main:int}
Let $\alpha > 1$ and $\bsgamma$ be a sequence of weights. Consider multivariate integration $\INT = (\INT_s)_{s \ge 1}$ for the weighted Hermite spaces $\cH_R$, $R\in \{r_{s,\alpha,\bsgamma},\rho_{s,\alpha,\bsgamma},\psi_{s,\alpha,\bsgamma}\}$ for $s \in \NN$. Then we have the following sufficient conditions:
\begin{enumerate}
\item SPT holds if $$\sum_{j =1}^{\infty} \gamma_j < \infty$$ (which is equivalent to $s_{\bsgamma} \le 1$). In this case the exponent of SPT satisfies 
\begin{equation*}		
\tau^{\ast}(\Lambda^{\mathrm{std}}) \le 2 \max\left(\frac{1}{\alpha},s_{\bsgamma}\right) 
\end{equation*}
\item PT holds if 
\begin{equation*}	
\limsup_{s \to \infty} \frac1{\ln s} \sum_{j=1}^s \gamma_j < \infty.
\end{equation*}
\item WT holds if
\begin{equation*}
\lim_{s \to \infty} \frac1{s} \sum_{j=1}^s \gamma_j = 0.
\end{equation*}
\item For $\sigma\in (0,1]$ $(\sigma,\tau)$-WT holds if
\begin{equation*}
\lim_{s \to \infty} \frac1{s^\sigma} \sum_{j=1}^s \gamma_j = 0.
\end{equation*}
\item UWT holds if 
\begin{equation*}
\lim_{s \to \infty} \frac1{s^\sigma} \sum_{j=1}^s \gamma_j = 0 	\quad \text{for all } \sigma \in (0,1].
\end{equation*}
\end{enumerate}
\end{theorem} 

\begin{proof}
Using \eqref{vgl:n:int:app} one can transfer our results about tractability of the $L_2$-approximation problem for standard information to the integration problem in $\cH_R$.
\end{proof}

\begin{remark}\rm
Note that item~1 of the theorem yields an improvement over the upper bound on the exponent of SPT in \cite{IL} from 2 to $2 \max(1/\alpha,s_{\bsgamma})$.
\end{remark}

\begin{remark}\rm
Like for the approximation problem using exclusively standard information, also for the integration problem some questions remain open. These comprise of the quest for necessary conditions for the respective notions of tractability, for necessary and sufficient conditions for QPT and results for the case of smoothness $\alpha=1$.
\end{remark}

If we restrict ourselves to linear algorithms of the form 
\begin{equation}\label{alg:lin}
A_{n,s}^{{\rm int}}(f)=\sum_{i=1}^{n} w_i f(\bsx_i)
\end{equation}
with $n \in \NN$, nodes $\bsx_1,\bsx_2,\ldots,\bsx_n$ in $\RR^s$ and {\it non-negative} integration weights $w_1,w_2,\ldots,w_n$ we can show that the sufficient conditions for tractability are even necessary.  This method has been used by Sloan and Wo\'{z}niakowski in \cite{slowo01} in the context of numerical integration in Korobov spaces. 

We introduce a restricted version of the information complexity by taking into account only linear algorithms with non-negative weights. Define, for $\varepsilon \in (0,1)$ and $s \in \NN$, the quantity 
\begin{eqnarray*}
n^{{\rm lin, pos}}(\varepsilon, \INT_s):=\min\{n \in \NN  & : & \exists A_{n,s}^{{\rm int}}\ \mbox{of the form \eqref{alg:lin} with non-negative weights,}\\
&&\mbox{such that }\ e^{{\rm int}}(A_{n,s}^{{\rm int}},\cH_{r_{s,\alpha,\bsgamma}}) \le \varepsilon\}.
\end{eqnarray*}
Obviously, $n(\varepsilon, \INT_s) \le n^{{\rm lin, pos}}(\varepsilon, \INT_s)$.

\begin{theorem}\label{thm:nec_tract_int}
Let $\alpha \ge 1$ and $\bsgamma$ be a sequence of weights. Consider multivariate integration $\INT = (\INT_s)_{s \ge 1}$ for the weighted Hermite spaces $\cH_R$, $R \in \{r_{s,\alpha,\bsgamma},\rho_{s,\alpha,\bsgamma},\psi_{s,\alpha,\bsgamma}\}$ for $s \in \NN$, but restrict to  the class of linear algorithms of the form \eqref{alg:lin} with non-negative weights. Then we have the following necessary conditions:
\begin{enumerate}
\item SPT implies $$\sum_{j =1}^{\infty} \gamma_j < \infty.$$
\item PT implies 
\begin{equation*}	
\limsup_{s \to \infty} \frac1{\ln s} \sum_{j=1}^s \gamma_j < \infty.
\end{equation*}
\item WT implies
\begin{equation*}
\lim_{s \to \infty} \frac1{s} \sum_{j=1}^s \gamma_j = 0.
\end{equation*}
\item For $\sigma\in (0,1]$ $(\sigma,\tau)$-WT implies
\begin{equation*}
\lim_{s \to \infty} \frac1{s^\sigma} \sum_{j=1}^s \gamma_j = 0.
\end{equation*}
\item UWT implies
\begin{equation*}
\lim_{s \to \infty} \frac1{s^\sigma} \sum_{j=1}^s \gamma_j = 0 	\quad \text{for all } \sigma \in (0,1].
\end{equation*}
\end{enumerate}
\end{theorem} 

The proof of Theorem~\ref{thm:nec_tract_int} is based on the following proposition, which is an analogy to \cite[Theorem~4]{slowo01}, that  applies to Korobov spaces.

\begin{proposition}\label{pr:lowbd_int}
For every linear algorithm $A_{n,s}^{{\rm int}}$ of the form \eqref{alg:lin} with non-negative integration weights we have $$\left(e^{{\rm int}}(A_{n,s}^{{\rm int}},\cH_{r_{s,\alpha,\bsgamma}})\right)^2 \ge 1-\frac{n}{\prod_{j=1}^s \left(1+\gamma_j c_{\omega}\right)},$$ where $c_{\omega}:=(1-\sqrt{1-\omega^2})/\sqrt{1-\omega^2}>0$ and $\omega:=3^{-\alpha/3}$. In particular, $$n^{{\rm lin, pos}}(\varepsilon, \INT_s) \ge (1-\varepsilon^2) \prod_{j=1}^s (1+\gamma_j c_\omega).$$
\end{proposition}

\begin{proof}
We define a further reproducing kernel Hilbert space based on Hermite polynomials. For $\omega \in (0,1)$ we let  $$K_{\phi_{s,\omega,\bsgamma}}(\bsx,\bsy):= \sum_{\bsk \in \NN_0^s} \phi_{s,\omega,\bsgamma}(\bsk) H_{\bsk}(\bsx) H_{\bsk}(\bsy),$$ where now the used Fourier weights are $R(\bsk)=\phi_{s,\omega,\bsgamma}(\bsk) := \prod_{j=1}^s \phi_{\omega,\gamma_j}(k_j)$ with 
\begin{equation*}
	\phi_{\omega,\gamma}(k)
	:=
	\left\{\begin{array}{ll}
		1 & \text{for } k=0, \\[0.5em]
		\gamma \omega^k  & \text{for } k \ge 1.
	\end{array}\right.
\end{equation*}
Let $\cH_{\phi_{s,\omega,\bsgamma}}$ denote the corresponding reproducing kernel Hilbert space with inner product and norm 
\begin{equation*}
	\langle f,g \rangle_{\cH_{\phi_{s,\omega,\bsgamma}}}
	:=
	\sum_{\bsk \in \NN_0^s} \frac1{\phi_{s,\omega,\bsgamma}(\bsk)} \, \widehat{f}(\bsk) \, \widehat{g}(\bsk)
	\quad \text{and} \quad
	\|f\|_{\phi_{s,\omega,\bsgamma}}
	=
	\sqrt{\langle f,f \rangle_{\cH_{\phi_{s,\omega,\bsgamma}}}}
	,
\end{equation*}
respectively. The space $\cH_{\phi_{s,\omega,\bsgamma}}$ and integration therein has been studied already in \cite{IL}. In particular, in \cite[Proposition~3.7]{IL}  it is shown that the functions from the space $\cH_{\phi_{s,\omega,\bsgamma}}$ are analytic functions.

For $\alpha \ge 1$ choose $\omega=\omega(\alpha) \in (0,1)$ such that we have $$\frac{1}{k^{\alpha}} \ge \omega^{k} \quad \mbox{for all $k \in \NN$.}$$ For example $\omega:=\min_{k \ge 1} k^{-\alpha/k}=3^{-\alpha/3}$ is a suitable choice. Then we have $$r_{s,\alpha,\bsgamma}(\bsk) \ge \phi_{s,\omega,\gamma}(\bsk)  \quad \mbox{for all $\bsk \in \NN_0^s$}$$ and hence, for every $f \in \cH_{\phi_{s,\omega,\bsgamma}}$ we have $$\|f\|_{r_{s,\alpha,\bsgamma}} \le \|f\|_{\phi_{s,\omega,\bsgamma}}.$$ This shows that the space $\cH_{\phi_{s,\omega,\bsgamma}}$ is continuously embedded in the space $\cH_{r_{s,\alpha,\bsgamma}}$ and the norm of the embedding operator is at most 1. Therefore, integration in $\cH_{\phi_{s,\omega,\bsgamma}}$ is not harder than in the space $\cH_{r_{s,\alpha,\bsgamma}}$. This implies that for every algorithm, and we restrict ourselves to linear algorithms $A_{n,s}^{{\rm int}}$ with non-negative integration weights like in \eqref{alg:lin} in the following,  the integration errors in $\cH_{\phi_{s,\omega,\bsgamma}}$ and in $\cH_{r_{s,\alpha,\bsgamma}}$, respectively, are related as $$e^{{\rm int}}(A_{n,s}^{{\rm int}},\cH_{\phi_{s,\omega,\bsgamma}}) \le e^{{\rm int}}(A_{n,s}^{{\rm int}},\cH_{r_{s,\alpha,\bsgamma}}).$$ 

Now we consider $e^{{\rm int}}(A_{n,s}^{{\rm int}},\cH_{\phi_{s,\omega,\bsgamma}})$. Using a well-known formula for the squared integration error of linear algorithms in reproducing kernel Hilbert spaces (see, e.g., \cite[Exercise~3.15]{LP14}) we obtain
\begin{eqnarray*}
\left(e^{{\rm int}}(A_{n,s}^{{\rm int}},\cH_{\phi_{s,\omega,\bsgamma}})\right)^2 & = & \int_{\RR^s} \int_{\RR^s} K_{\phi_{s,\omega,\bsgamma}} (\bsx,\bsy) \varphi_s(\bsx)\varphi_s(\bsy) \rd \bsx \rd \bsy\\
&& - 2 \sum_{i=1}^n w_i \int_{\RR^s} K_{\phi_{s,\omega,\bsgamma}} (\bsx,\bsx_i) \varphi_s(\bsx) \rd \bsx \\
&&+\sum_{i,\ell=1}^n w_i w_{\ell} K_{\phi_{s,\omega,\bsgamma}} (\bsx_i,\bsx_{\ell}).
\end{eqnarray*}
It is easy to see (or consult \cite[p.~191]{IL}) that $$\int_{\RR^s} \int_{\RR^s} K_{\phi_{s,\omega,\bsgamma}} (\bsx,\bsy) \varphi_s(\bsx)\varphi_s(\bsy) \rd \bsx \rd \bsy=1$$ and $$\int_{\RR^s} K_{\phi_{s,\omega,\bsgamma}} (\bsx,\bsx_i) \varphi_s(\bsx) \rd \bsx=1 \quad \mbox{for all $i \in \{1,2,\ldots,n\}$.}$$ Therefore we obtain
\begin{eqnarray}\label{fo:err:lin1}
\left(e^{{\rm int}}(A_{n,s}^{{\rm int}},\cH_{\phi_{s,\omega,\bsgamma}})\right)^2  =1 - 2 \sum_{i=1}^n w_i  +\sum_{i,\ell=1}^n w_i w_{\ell} K_{\phi_{s,\omega,\bsgamma}} (\bsx_i,\bsx_{\ell}).
\end{eqnarray}
From Mehler's formula (see \cite{szeg}), which states that for every $x,y\in \RR$ and every $\omega \in (-1,1)$ we have
\[
\sum_{k=0}^\infty H_k(x)H_k(y) \omega^k
=\frac{1}{\sqrt{1-\omega^2}}\exp\left(\frac{\omega x y }{1+\omega}-\frac{\omega^2 (x-y)^2 }{2(1-\omega^2)}\right),
\]
one can derive that 
$$K_{\phi_{s,\omega,\bsgamma}}(\bsx,\bsy)=\prod_{j=1}^s \left(1-\gamma_j+\gamma_j \frac{1}{\sqrt{1-\omega^2}} \exp\left(\frac{\omega x_j y_j}{1+\omega} - \frac{\omega^2 (x_j-y_j)^2}{2(1-\omega^2)} \right) \right).$$
This shows, in particular, that the kernel $K_{\phi_{s,\omega,\bsgamma}}$ is non-negative. Since also the integration weights $w_i$ are non-negative, we deduce from \eqref{fo:err:lin1} by neglecting the non-diagonal terms in the double-sum that
\begin{eqnarray}\label{fo:err:lin2}
\left(e^{{\rm int}}(A_{n,s}^{{\rm int}},\cH_{\phi_{s,\omega,\bsgamma}})\right)^2  \ge 1 - 2 \sum_{i=1}^n w_i  +\sum_{i=1}^n w_i^2 K_{\phi_{s,\omega,\bsgamma}} (\bsx_i,\bsx_i).
\end{eqnarray}
Now, for $i \in \{1,2,\ldots,n\}$ we have
\begin{eqnarray*}
K_{\phi_{s,\omega,\bsgamma}}(\bsx_i,\bsx_i) & = &\prod_{j=1}^s \left(1-\gamma_j+\gamma_j \frac{1}{\sqrt{1-\omega^2}} \exp\left(\frac{\omega x_{i,j}^2}{1+\omega}\right) \right)\\
& \ge & \prod_{j=1}^s \left(1-\gamma_j+\gamma_j \frac{1}{\sqrt{1-\omega^2}}\right)\\
& = & \prod_{j=1}^s \left(1+\gamma_j c_{\omega}\right),
\end{eqnarray*}
where $x_{i,j}$ is the $j$-th component of $\bsx_i$ and $c_{\omega}:=(1-\sqrt{1-\omega^2})/\sqrt{1-\omega^2}>0$. Inserting this estimate into \eqref{fo:err:lin2} we get
\begin{eqnarray}\label{fo:err:lin3}
\left(e^{{\rm int}}(A_{n,s}^{{\rm int}},\cH_{\phi_{s,\omega,\bsgamma}})\right)^2  \ge 1 - 2 \sum_{i=1}^n w_i  +\sum_{i=1}^n w_i^2 \prod_{j=1}^s \left(1+\gamma_j c_{\omega}\right).
\end{eqnarray}
Next, set $\beta:=(\sum_{i=1}^n w_i^2)^{1/2}$ and observe that by the Cauchy-Schwarz inequality we have $\sum_{i=1}^n w_i \le \sqrt{n} \beta$. Thus we my conclude from \eqref{fo:err:lin3} that 
\begin{eqnarray}\label{fo:err:lin4}
\left(e^{{\rm int}}(A_{n,s}^{{\rm int}},\cH_{\phi_{s,\omega,\bsgamma}})\right)^2  \ge 1 - 2 \sqrt{n} \beta + \beta^2  \prod_{j=1}^s \left(1+\gamma_j c_{\omega}\right).
\end{eqnarray}
Minimizing the expression on the right-hand side of \eqref{fo:err:lin4} with respect to $\beta$ we obtain that $$\left(e^{{\rm int}}(A_{n,s}^{{\rm int}},\cH_{\phi_{s,\omega,\bsgamma}})\right)^2  \ge 1-\frac{n}{\prod_{j=1}^s \left(1+\gamma_j c_{\omega}\right)}.$$ From here the upper bound on $n^{{\rm lin, pos}}(\varepsilon, \INT_s)$ follows immediately.
\end{proof}

Now we can give the proof of Theorem~\ref{thm:nec_tract_int}.

\begin{proof}[Proof of Theorem~\ref{thm:nec_tract_int}]
Again it suffices to prove the result for $R=r_{s,\alpha,\bsgamma}$. We will use Proposition~\ref{pr:lowbd_int} and arguments from \cite[Proof of Theorem~5]{slowo01}. 

Assume that the weights are bounded from below by some positive number $\gamma_\ast$, i.e. $\gamma_j \ge \gamma_\ast>0$ for all $j \in \NN$. Then it follows from Proposition~\ref{pr:lowbd_int} that $$n^{{\rm lin, pos}}(\varepsilon, \INT_s) \ge (1-\varepsilon^2) (1+\gamma_\ast c_\omega)^s.$$ Thus $n^{{\rm lin, pos}}(\varepsilon, \INT_s)$ grows exponentially fast in $s$ and hence we cannot have any form of tractability. Thus, if we have some form of tractability, then we must also have  $\lim_{j \rightarrow \infty}\gamma_j=0$.

Now suppose that we have $\lim_{j \rightarrow \infty}\gamma_j=0$ but $\sum_{j=1}^{\infty}\gamma_j = \infty$. For $\lim_{j \rightarrow \infty}\gamma_j=0$ it is a well-known fact that
\begin{equation}\label{eq:equivgrweigt}
\prod_{j=1}^s(1+\gamma_j c_\omega)= \Theta\left(\exp\left(c_\omega \sum_{j=1}^s \gamma_j\right)\right).
\end{equation}
Then it follows from Proposition~\ref{pr:lowbd_int} and Equation~\eqref{eq:equivgrweigt} that $\lim_{s \rightarrow \infty}n^{{\rm lin, pos}}(\varepsilon, \INT_s) = \infty$ and this contradicts SPT. Thus $\sum_{j=1}^{\infty}\gamma_j < \infty$ is a necessary condition for SPT.

Suppose next that we have $\lim_{j \rightarrow \infty}\gamma_j=0$ but $\limsup_{s \rightarrow \infty} (1/\log s) \sum_{j=1}^s \gamma_j=\infty$. Since $$\prod_{j=1}^s(1+\gamma_j c_\omega)= \Theta\left(s^{c_\omega (1/\log s) \sum_{j=1}^s \gamma_j} \right),$$ it follows from Proposition~\ref{pr:lowbd_int} that $n^{{\rm lin, pos}}(\varepsilon, \INT_s)$ goes to infinity faster than any power of $s$ and this contradicts PT. Thus  $\limsup_{s \rightarrow \infty} (1/\log s) \sum_{j=1}^s \gamma_j< \infty$ is a necessary condition for PT.

Finally, assume that for $\sigma \in (0,1]$ we have $$\lim_{s+\varepsilon^{-1} \rightarrow \infty} \frac{\log n^{{\rm lin, pos}}(\varepsilon, \INT_s)}{s^{\sigma}+\varepsilon^{-\tau}}=0.$$ Then it follows from Proposition~\ref{pr:lowbd_int} and Equation~\eqref{eq:equivgrweigt}  that  $$\lim_{s \rightarrow \infty}\frac{1}{s^{\sigma}} \sum_{j=1}^s \gamma_j =0.$$ This implies the necessary conditions for the three WT notions. 
\end{proof}

\section{Remarks on integration in the anchored space}

In \cite{WW02} Wasilkowski and Wo\'{z}niakowski studied $L_2$-approximation and integration over unbounded domains. The underlying function space in this work is a more general version of the reproducing kernel Hilbert space with kernel $L$ from \eqref{prop:decomposable}. Choosing $\psi=\varphi^{1/2}$ and $\omega=\varphi$ in \cite{WW02} corresponds exactly to the setting of the present work.    

Unfortunately, the results from \cite{WW02} concerning tractability of $L_2$-approximation cannot be transferred to our setting here (see Section~\ref{subsec:anch}), since \cite[Theorems~1 and 2]{WW02} require the assumption \cite[Eq.~(17)]{WW02} which is $\int_{\RR} (\sqrt{\omega(x)}/\psi(x))^{1/\alpha} \rd x < \infty$, but which is obviously not satisfied in our case where $\psi=\varphi^{1/2}$ and $\omega=\varphi$.

The results about integration in \cite{WW02} do not require this assumption.  This means we can transfer them directly into our setting in order to obtain ``if and only if''-conditions for numerical integration in the anchored space $\cH_{\pitchfork,s,\alpha,\bsgamma}$. This has already been done in  \cite[Sec.~12.5.1]{NW10} (in a slightly different but equivalent formulation). The following result is basically \cite[Corollary~12.8]{NW10} (which we extend by results about $(\sigma,\tau)$-WT and UWT). We stress that here we also have necessary conditions thanks to the fact that the kernel $K_{\pitchfork,\alpha,\gamma}$ contains the decomposable part $L$.

\begin{theorem}
Let $\alpha > 1$ and $\bsgamma$ be a sequence of weights. Consider multivariate integration $\INT = (\INT_s)_{s \ge 1}$ for the weighted anchored spaces $\cH_{\pitchfork,s,\alpha,\bsgamma}$ for $s \in \NN$.  Then we have:
\begin{enumerate}
\item SPT holds if and only if $$\sum_{j =1}^{\infty} \gamma_j < \infty.$$
\item PT holds if and only if 
\begin{equation*}	
\limsup_{s \to \infty} \frac1{\ln s} \sum_{j=1}^s \gamma_j < \infty.
\end{equation*}
\item WT holds if and only if 
\begin{equation*}
\lim_{s \to \infty} \frac1{s} \sum_{j=1}^s \gamma_j = 0.
\end{equation*}
\item For $\sigma\in (0,1]$ $(\sigma,\tau)$-WT holds if and only if 
\begin{equation*}
\lim_{s \to \infty} \frac1{s^\sigma} \sum_{j=1}^s \gamma_j = 0.
\end{equation*}
\item UWT holds if and only if 
\begin{equation*}
\lim_{s \to \infty} \frac1{s^\sigma} \sum_{j=1}^s \gamma_j = 0 	\quad \text{for all } \sigma \in (0,1].
\end{equation*}
\end{enumerate}
\end{theorem}

\begin{appendix}
\section{Appendix: The proof of Theorem~\ref{pr:intrep}}\label{appA}

For the Gaussian ANOVA space we know the Hermite expansion of the reproducing kernel, namely, in dimension 1 and for a generic weight $\gamma>0$, 
\begin{equation*}
K_{r_{\alpha,\gamma}}(x,y)=\sum_{k \in \NN_0} r_{\alpha,\gamma}(k) H_{k}(x) H_{k}(y).
\end{equation*}
Now we derive the integral representation presented in Theorem~\ref{pr:intrep}.

The starting point is the weighted Gaussian ANOVA norm from \eqref{fo_norm1a} given by 
\[
\|f\|_{r_{1,\alpha,\gamma}}^2=\left(\int_\RR f(y)\varphi(y)\rd y\right)^2+\frac{1}{\gamma}\sum_{k=1}^{\alpha-1}\left(\int_\RR f^{(k)}(y)\varphi(y)\rd y\right)^2+\frac{1}{\gamma}\int_\RR (f^{(\alpha)}(y))^2\varphi(y) \rd y
\]
on the space of functions 
\[
\cH_\alpha=\left\{f\colon \RR\to \RR\colon f^{(\alpha-1)} \text{ exists and is abs.~continuous}, \int_\RR |f^{(\alpha)}(y)|^2\varphi(y) \rd y<\infty\right\}.
\]
The space $\cH_\alpha$ decomposes into the orthogonal subspaces 
\begin{align*}
\cH_{1,\alpha}&:=\left\{f\colon \RR\to \RR\colon f^{(\alpha-1)} \text{ exists and is abs.~continuous}, \int_\RR |f^{(\alpha)}(y)|^2\varphi(y) \rd y=0\right\},\\
\cH_{2,\alpha}&:=\left\{f\colon \RR\to \RR\colon f^{(\alpha-1)} \text{ exists and is abs.~continuous}, \int_\RR |f^{(\alpha)}(y)|^2\varphi(y) \rd y<\infty,\right.\\ 
&\hspace{17em}\left. \int_\RR f^{(k)}(y)\varphi(y)\rd y=0, \, k\in\{0,\ldots,\alpha-1\}\right\},
\end{align*}
and therefore, using property (7) from \cite[Section~2]{aronszajn50}, the reproducing kernel $K_\alpha$ of the space $(\cH_\alpha,\|\cdot\|_{r_{1,\alpha,\gamma}})$ is of the form $K_\alpha=K_{1,\alpha}+K_{2,\alpha}$, where $K_{j,\alpha}$ is a reproducing kernel for $\cH_{j,\alpha}$, $j \in \{1,2\}$. Clearly, $\cH_{1,\alpha}$ consists precisely of the polynomials of degree smaller than $\alpha$ and therefore every $f_1\in \cH_{1,\alpha}$ can be written as 
\[
f_1(x)=\sum_{k=0}^{\alpha-1}  \widehat{f}_1(k)H_k(x), 
\]
such that, for $j \in \{0,1,\ldots,\alpha-1\}$, 
\[
\int_\RR f_1^{(j)}(y)\varphi(y) \rd y=\sum_{k=j}^{\alpha-1}  \sqrt{\frac{k!}{(k-j)!}}\, \widehat{f}_1(k)\int_\RR H_{k-j}(y)\varphi(y) \rd y
=\sqrt{j!}\, \widehat{f}_1(j)
\]
and therefore 
\begin{align*}
f_1(x)=&\sum_{j=0}^{\alpha-1}(j!)^{-1/2} \int_\RR f^{(j)}(y)\, \varphi(y) \rd y  \, H_j(x)\\
=&\int_\RR f(y) \varphi(y) \rd y \ \int_\RR K_{1,\alpha}(x,y) \varphi(y)\rd y\\
&+ \frac{1}{\gamma} \sum_{j=1}^{\alpha-1} \int_\RR f^{(j)}(y)\, \varphi(y) \rd y \ \int_\RR \left(\frac{\partial^j}{\partial y^j}
K_{1,\alpha}(x,y)\right) \varphi(y)\rd y
\end{align*}
with $K_{1,\alpha}(x,y)=1+\sum_{k=1}^{\alpha-1} \frac{\gamma}{k!}H_k(x)H_k(y)$. Here we used that the $j$-th derivative of $H_k$ equals $$H_k^{(j)}(y) =\left\{
\begin{array}{ll}
\sqrt{\frac{k!}{(k-j)!}}\, H_{k-j}(y) & \mbox{ if $k \ge j$,}\\[0.5em]
0 & \mbox{ otherwise,}
\end{array}\right.$$ from which we obtain that $$H_j(x)=\sqrt{j!} \int_\RR \left(\frac{\partial^j}{\partial y^j} K_{1,\alpha}(x,y)\right) \varphi(y)\rd y.$$

\medskip

We proceed to compute $K_{2,\alpha}$.
Here and in the following we write $\Phi(y):=\int_{-\infty}^y \varphi(\eta)\rd \eta$
and 
\[
\vartheta(x,y):=1_{(-\infty,x]}(y)\Phi(y)-1_{(x,\infty)}(y)\Phi(-y)\,.
\]
Recall that for $y\le -1$ we have 
\begin{align*}
0\le \Phi(y)&=\int_{-\infty}^{y} \varphi(\eta)\rd \eta\le \int_{-\infty}^{y} (-\eta)\varphi(\eta) \rd \eta
=\int_{-\infty}^{y} \varphi'(\eta)\rd \eta=\varphi(y)
\end{align*}
and that therefore also $0\le \Phi(-y)\le \varphi(-y)=\varphi(y)$ for $y>1$
so that $$\int_\RR 1_{(-\infty,x]}(y)\Phi(y)\rd y \quad \mbox{ and } \quad \int_\RR 1_{(x,\infty)}(y)\Phi(-y)\rd y$$ are real numbers.

\begin{lemma}\label{lem:primitive-function}
Let $h\colon \RR\to\RR$ be measurable with $\int_\RR (h(y))^2 \varphi(y) \rd y <\infty$. Then 
$g\colon \RR\to\RR$ with 
\[
g(x):=\int_\RR h(y) \vartheta(x,y)\rd y\quad \mbox{ for } x\in\RR,
\]
is the unique absolutely continuous function with $g'=h$ a.e. and
$\int_\RR g(y)\varphi(y)\rd y=0$. 
\end{lemma}

\begin{proof}
Since we may write 
\[
g(x):=\int_\RR h(y) \vartheta(x,y) \rd y=\int_{-\infty}^x h(y) \Phi(y) \rd y-\int_x^{\infty}h(y)\Phi(-y) \rd y,
\]
and the integrals exist since $$\int_{-\infty}^{-1} |h(y)| \Phi(y) \rd y
\le \int_{-\infty}^{-1} |h(y)| \varphi(y) \rd y\le \left(\int_\RR (h(y))^2 \varphi(y) \rd y\right)^{1/2}< \infty,$$
it is clear that $g$ is absolutely continuous. Differentiating gives a.e.
\[
g'(x)=h(x) \Phi(x)+h(x)\Phi(-x)=h(x) \Phi(x)+h(x)(1-\Phi(x))=h(x)\,.
\]

Next we integrate $g$ with respect to the weight $\varphi$ and use Fubini's theorem to get
\begin{align*}
\int_\RR g(x)\varphi(x) \rd x
&=\int_\RR \int_{-\infty}^x h(y) \Phi(y)\rd y\,\varphi(x)\rd x-\int_\RR \int_x^{\infty}h(y)\Phi(-y) \rd y\,\varphi(x)\rd x\\
&=\int_\RR \int_y^{\infty} h(y) \Phi(y) \varphi(x) \rd x \rd y-\int_\RR \int_{\infty}^yh(y)\Phi(-y) \varphi(x) \rd x \rd y\\
&=\int_\RR  h(y) \Phi(y) \Phi(-y) \rd y-\int_\RR h(y)\Phi(-y)\Phi(y) \rd y=0.
\end{align*}

Finally let $g_1$ be an arbitrary absolutely continuous function with $g_1'=h$ a.e.~and $\int_\RR g'(y)\varphi(y) \rd y=0$. Then $g_1(x)=\int_0^x h(y) \rd y+c_1$ a.e.
On the other hand, $g(x)=\int_0^x h(y) \rd y+c_2$, so $g_1(x)=g(x)+c_3$ a.e. 
But since $0=\int_\RR g_1(x)\varphi(x) \rd x=\int_\RR g(x)\varphi(x) \rd x + c_3=c_3$, we have $g_1=g$ a.e.
\end{proof}

Next we compute the reproducing Kernel $K_{2,\alpha}$ for $\alpha=1$. For every $f\in \cH_{2,1}$, i.e., with 
$$\int_\RR f(y)\varphi(y)\rd y=0 \quad \mbox{ and } \quad \int_\RR (f'(y))^2 \varphi(y) \rd y<\infty$$
we have from the reproducing property of the kernel $K_{2,1}$ that
\[
f(x)=\frac{1}{\gamma}\int_\RR f'(y) \frac{\partial}{\partial y}K_{2,1}(x,y) \varphi(y) \rd y.
\]
On the other hand we know from Lemma~\ref{lem:primitive-function} that 
\(
f(x)=\int_\RR f'(y) \vartheta(x,y) \rd y, 
\)
so 
\[
\int_\RR f'(y) \left(\frac{1}{\gamma} \frac{\partial}{\partial y}K_{2,1}(x,y) \varphi(y)-\vartheta(x,y)\right) \rd y=0.
\]
Since this holds in particular if $f'$ is the indicator function of an arbitrary measurable set, we conclude $$\frac{\partial}{\partial y}K_{2,1}(x,y)= \gamma \varphi(y)^{-1}\vartheta(x,y)$$ 
for a.e.~$y\in \RR$. 
Since $y\mapsto K_{2,1}(x,y)$ is an element of $\cH_{2,1}$, we need to have $\int_\RR K_{2,1}(x,y)\varphi(y) \rd y=0$. Again from Lemma~\ref{lem:primitive-function} we derive the integral representation
\[
K_{2,1}(x,y)=\int_\RR\frac{\partial}{\partial y}K_{2,1}(x,\eta)\vartheta(y,\eta) \rd\eta 
=\gamma \int_\RR\varphi(\eta)^{-1}\vartheta(x,\eta)\vartheta(y,\eta)\rd\eta\,, 
\]
Which finishes the proof of Theorem \ref{pr:intrep} for the case $\alpha=1$.
Note that 
\begin{align*}
\vartheta(x,\eta)\vartheta(y,\eta)
&=\left(1_{(-\infty,x]}(\eta)\Phi(\eta)-1_{(x,\infty)}(\eta)\Phi(-\eta)\right)
\left(1_{(-\infty,y]}(\eta)\Phi(\eta)-1_{(y,\infty)}(\eta)\Phi(-\eta)\right)\\
&=1_{(-\infty,\min(x,y)]}(\eta)(\Phi(\eta))^2+1_{(\max(x,y),\infty)}(\eta)(\Phi(-\eta))^2
\\&\quad 
-1_{(\min(x,y),\max(x,y)]}(\eta) \Phi(\eta)\Phi(-\eta).
\end{align*}
So $K_{2,1}$ can be written in terms of the primitive functions of $$\varphi^{-1}\Phi^2,\ \varphi^{-1}(1-\Phi)^2,\ \mbox{and}\ \varphi^{-1}\Phi(1-\Phi).$$ 

Next we compute $K_{2,\alpha}$ for $\alpha=2$. If $f\in \cH_{2,2}$, we obtain from the reproducing property of the kernel $K_{2,2}$ that
\[
f(x)=\frac{1}{\gamma} \int_\RR f''(y)\frac{\partial^2}{\partial y^2}K_{2,2}(x,y) \rd y
\]
so that \[
f'(x)=\frac{1}{\gamma} \int_\RR f''(y)\frac{\partial^3}{\partial x\partial y^2}K_{2,2}(x,y)\varphi(y)\rd y.
\] 
On the other hand, if $f\in \cH_{2,2}$ then $f'\in \cH_{2,1}$, and therefore
\[f'(x)=\frac{1}{\gamma} \int_\RR f''(y)\frac{\partial}{\partial y}K_{2,1}(x,y)\varphi(y) \rd y.\] Thus we have 
\[
\frac{1}{\gamma} \int_\RR f''(y)\left(\frac{\partial^3}{\partial x\partial y^2}K_{2,2}(x,y)-\frac{\partial}{\partial y}K_{2,1}(x,y)\right)\varphi(y) \rd y =0
\]
from which we obtain
$\frac{\partial^3}{\partial x\partial y^2}K_{2,2}=\frac{\partial}{\partial y}K_{2,1}$.
Now $$\frac{\partial^2}{\partial x\partial y}K_{2,2}(x,y)=K_{2,1}(x,y)+c_1(x)$$ and, since $K_{2,2}$ is symmetric and $\frac{\partial^2}{\partial x\partial y}K_{2,2}$ is continuous, $\frac{\partial^2}{\partial x\partial y}K_{2,2}=\frac{\partial^2}{\partial y\partial x}K_{2,2}$, so $\frac{\partial^2}{\partial x\partial y}K_{2,2}$ is also symmetric. From this it follows that 
$c_1$ is actually constant, $c_1(x)=c_1$. 

Since the function $y\mapsto K_{2,2}(x,y)$ is an element of $\cH_{2,2}$ for every $x\in \RR$, we get
\[
\int_\RR\frac{\partial}{\partial y}K_{2,2}(x,y)\varphi(y) \rd y=0 \quad \mbox{for every $x \in \RR$,}
\]
so that also 
\begin{eqnarray*}
0 & = & \frac{\partial}{\partial x}\int_\RR\frac{\partial}{\partial y}K_{2,2}(x,y)\varphi(y)\rd y \\
& = &  \int_\RR \frac{\partial^2}{\partial x\partial y} K_{2,2}(x,y)\varphi(y) \rd y \\
& = & \int_\RR (K_{2,1}(x,y)+c_1)\varphi(y) \rd y\\
& = & c_1.
\end{eqnarray*}
So, actually $c_1=0$, and hence $$\frac{\partial^2}{\partial x\partial y}K_{2,2}(x,y)=K_{2,1}(x,y).$$
Now, integrating with respect to $x$ another time, we get using Lemma~\ref{lem:primitive-function} once more that
\begin{equation}\label{eq:kern22}
\frac{\partial}{\partial y}K_{2,2}(x,y)=\int_\RR K_{2,1}(\xi,y)\vartheta(x,\xi) \rd \xi+c_2(y),
\end{equation}
so 
\begin{equation*}
\int_\RR \frac{\partial}{\partial y}K_{2,2}(x,y)\varphi(x)\rd x=\int_\RR K_{2,1}(\xi,y)\int_\RR\vartheta(x,\xi)\varphi(x)\rd x \rd\xi+c_2(y).
\end{equation*}
But
\begin{align*}
\int_\RR \vartheta(x,y)\varphi(x)\rd x
&=\int_\RR 1_{(-\infty,x]}(y)\varphi(x) \rd x \,\Phi(y)-\int_\RR 1_{(x,\infty)}(y)\varphi(x) \rd x\,\Phi(-y)\\
&=\Phi(-y)\Phi(y)-\Phi(y)\Phi(-y)=0
\end{align*}
and 
$$\int_\RR \frac{\partial}{\partial y}K_{2,2}(x,y)\varphi(x)\rd x=\frac{\partial}{\partial y} \int_\RR K_{2,2}(x,y)\varphi(x)\rd x=0$$
by symmetry of $K_{2,2}$ so $c_2(y)=0$.
Now, integrating \eqref {eq:kern22} (with $c_2(y)=0$) with respect to $y$ gives,
\[
K_{2,2}(x,y)=\int_\RR\int_\RR K_{2,1}(\xi,\eta)\vartheta(x,\xi)\vartheta(y,\eta) \rd \xi \rd \eta+c_3(x),
\]
and we see that
\[
0=\int_\RR K_{2,2}(x,y)\varphi(y) \rd y=\int_\RR\int_\RR K_{2,1}(\xi,\eta)\vartheta(x,\xi)\int_\RR\vartheta(y,\eta)\varphi(y) \rd y \rd\xi \rd\eta+c_3(x)=c_3(x).
\]
Thus we have found 
\begin{align*}
K_{2,2}(x,y) &= \int_\RR\int_\RR K_{2,1}(\xi,\eta)\vartheta(x,\xi)\vartheta(y,\eta) \rd\xi \rd\eta\\
&=\gamma \int_{\RR^3} \frac{1}{\varphi(s)}\vartheta(\xi,s)\vartheta(\eta,s)\vartheta(x,\xi)\vartheta(y,\eta) \rd s \rd \xi \rd \eta.
\end{align*}

By induction we get the integral representation of the general kernel 
\[
K_{2,\alpha}(x,y)=\gamma \int_{\RR^{2\alpha-1}} \frac{1}{\varphi(s)}\vartheta_\alpha(x,\xi_{\alpha-1},\ldots,\xi_1,s)\vartheta_\alpha(y,\eta_{\alpha-1},\ldots,\eta_1,s) \rd s \prod_{k=1}^{\alpha-1}(\rd \xi_k \rd\eta_k)
\]
with $
\vartheta_{n}(z_1,\ldots,z_{n+1}):=\prod_{k=1}^{n}\vartheta(z_{k},z_{k+1})
$
for $n\in \NN$. This finishes the proof of Theorem~\ref{pr:intrep}. \qed

\end{appendix}

\medskip

\noindent{\bf Acknowledgment.} We thank David Krieg for a valuable discussion, in particular for his help in proving the exponent of SPT in \eqref{cond_exp_spt_std}.

\end{document}